\documentclass[11pt]{amsart}
\usepackage[utf8]{inputenc}
\usepackage{ragged2e}
\usepackage{mathtools}
\usepackage{mathrsfs}
\usepackage{float}
\usepackage{tikz}
\usepackage{amsfonts}
\usepackage{amsmath}
\usepackage{enumerate}
\usepackage{amsthm}
\usepackage{amssymb}
\usepackage{dutchcal}
\usepackage{wasysym}
\usepackage[foot]{amsaddr}
\usepackage{bbm}
\usepackage[left=2.5cm,right=2.5cm,top=3.2cm,bottom=3cm]{geometry}
\usepackage{cleveref}
\usepackage{caption}
\crefformat{section}{\S#2#1#3}
\crefformat{subsection}{\S#2#1#3}

\usepackage[OT2,T1]{fontenc}
\DeclareSymbolFont{cyrletters}{OT2}{wncyr}{m}{n}
\DeclareMathSymbol{\Sha}{\mathalpha}{cyrletters}{"58}

\usepackage{pdfpages}
\usepackage[natbib, maxbibnames=6, maxcitenames=6]{biblatex}
\makeatletter
\renewcommand{\fnum@figure}{Figure \thefigure}

\newtheorem{teo}{Theorem}
\newtheorem{lem}{Lemma}[section]
\newtheorem{prop}[lem]{Proposition}
\newtheorem{thm}[lem]{Theorem}
\newtheorem{cor}{Corollary}[lem]
\newtheorem{defi}[lem]{Definition}
\newtheorem{es}[lem]{Example}
\newtheorem{con}{Conjecture}

\theoremstyle{remark}

\newtheorem{reml}[cor]{Remark}

\newtheorem{q}{Question}

\def\C{\mathbb{C}}
\def\c{\mathcal{C}}
\def\e{\mathcal{E}}
\def\R{\mathbb{R}}
\def\jac{\mathrm{Jac}}
\def\Q{\mathbb{Q}}
\def\Z{\mathbb{Z}}

\def\A{\mathbb{A}}
\def\P{\mathbb{P}}
\def\O{\mathcal{O}}

\def\F{\mathbb{F}}

\def\h{\mathbb{H}}
\def\H{\mathcal{H}}

\def\m{\mathscr{M}}

\def\PGL{\text{PGL}}

\def\Gal{\text{Gal}}

\def\Aut{\text{Aut}}

\title{The arithmetic of critical values II: critical elliptic curves}
\author{Francesco Naccarato}
\address{Department of Mathematics, ETH Zürich, Zürich CH-8092, Switzerland.}
\email{francesco.naccarato@math.ethz.ch}

\makeatletter
\newcounter{toclink}

\ifdefined\XeTeXversion
  \def\toc@dest#1{
    }
  \DeclareRobustCommand{\toc@link}[2]{\leavevmode
    
    #2}
\else
  \ifdefined\pdfextension
    \def\toc@dest#1{\pdfextension dest name {toc.#1} xyz}
    \DeclareRobustCommand{\toc@link}[2]{\leavevmode
      \pdfextension startlink attr{/Border [0 0 0]}
        goto name {toc.#1}
      #2\pdfextension endlink}
  \else
    \def\toc@dest#1{\pdfdest name {toc.#1} xyz}
    \DeclareRobustCommand{\toc@link}[2]{\leavevmode
      \pdfstartlink attr{/Border [0 0 0]}
        goto name {toc.#1}
      #2\pdfendlink}
  \fi
\fi

\let\toc@oldaddcontentsline\addcontentsline
\renewcommand{\addcontentsline}[3]{
  \def\toc@file{#1}\def\toc@target{toc}
  \ifx\toc@file\toc@target
    \stepcounter{toclink}
    \toc@dest{\arabic{toclink}}
    \toc@oldaddcontentsline{#1}{#2}{
      \toc@link{\arabic{toclink}}{#3}}
  \else
    \toc@oldaddcontentsline{#1}{#2}{#3}
  \fi}
\makeatother

\begin{document}

\begin{abstract}
    In this second chapter of the \textit{Arithmetic of critical values} series (ACV), we study certain double covers $E_f\to\P^1$ whose branch locus coincides with that of a quartic polynomial $f$. We give a direct proof of the fact, already shown non-constructively in ACV I, that the elliptic curves $E_f$ admit a $3$-isogeny. Our methods lead us to examine the Galois closure of $f:\P^1\to\P^1$, whose rich geometry yields a congruence between $E_f$ and a second elliptic curve $E'_f$. We study their arithmetic and show that they are linked by a Selmer companionship for infinitely many $f\in\Z[X]$, allowing us to exhibit elements in certain Tate--Shafarevich groups which are visible in an abelian surface. We also give some dynamical and Diophantine applications of our constructions, as well as new examples of Jacobians isogenous to a power of an elliptic curve.
\end{abstract}

\maketitle
\addtocontents{toc}{\protect\setlength{\parskip}{0pt}}
\vspace{-7mm}
\tableofcontents

\section{Introduction}\label{1}

We continue the study of the arithmetic of critical values started in ACV I \cite{nac}, this time focusing solely on the case of quartic polynomials. In our investigation, we encounter various families of elliptic curves which encode ramification data about our quartics, and we devote a large portion of the paper to their arithmetic geometry. While motivated by our previous work, the results in this paper are largely independent of it, making it suitable for the reader unfamiliar with ACV I. 

We now briefly recall the general framework for our study of critical values. For the moment, we work with polynomials (and, more generally, morphisms of curves) defined over a field $K$ of characteristic zero, which later on we will take as a number field; we also let $k=\overline K$. In our previous work, we defined the moduli space $H_d\to\text{Conf}_d(\P^1_k)$ of degree $d$ polynomials with given critical values, mainly by appealing to the theory of Hurwitz spaces. In general, the latter classify ($G$-)covers $Y\to X$ of a fixed curve $X$ (up to ($G$-)cover isomorphism) having given branch locus and local monodromy data. So, Hurwitz spaces relative to $X$ are naturally covers of the configuration space $\text{Conf}_d(X)$.

We also defined the \textit{reduced} Hurwitz space $\H_d\to\text{Conf}_d(\P^1_k)/\PGL_2(k)$ of degree $d$ polynomials. Reduced Hurwitz spaces classify covers as above up to the natural action of $\Aut(X)$---which naturally carries over to the branch locus. As it turns out, \textit{Morse} polynomials (see \cite[\S4.4]{serTopGal} or \S\ref{2} for the definition) of degree $d$ are the only covers $f:\P^1_k\to\P^1_k$ simply ramified at $d-1$ points, fully at one point and with monodromy group $S_d$, once we require the $d$-cycle to be at $f^{-1}(\infty)=\{\infty\}$. Therefore, after the $\PGL_2(k)$-quotient, we get (see \cite[\S 3.2]{nac}) a moduli space whose $k$-points are in bijection with linear equivalence classes of polynomials having given $\PGL_2(k)$-class of critical values, where we define:
\begin{defi}
    Two polynomials $f,g\in k[X]$ are \textbf{linearly equivalent} if there exist affine transformations $\alpha,\gamma\in\mathrm{Aff}(k)$ such that $f\circ\alpha=\gamma\circ g$.
\end{defi}
$\H_d$ turns out to be an affine algebraic variety of dimension $d-3$ (see \cite[\S1-2]{cad} for more details) and, under the aforementioned bijection, any $P\in\H_d(K)$ has a polynomial representative defined over an extension of $K$ of degree bounded by a function of $d$; this extension can be taken as $K$ itself when $d=4$ \cite[Theorem 4]{nac}. Coming from a cover of $\text{Conf}_4(\P^1_k)$, $\H_4$ is a modular curve. In ACV I \cite[Lemma 4.2]{nac}, we showed that \begin{align}
    \H_4\simeq_{\Q} Y_0(3).\label{h4}
\end{align}This has among its consequences the existence of infinitely many pairs of linearly inequivalent polynomials over $\Q$ with the same critical values \cite[Theorem 1]{nac}, which was the main application for the first chapter of the ACV series.

\subsection{Structure of the paper and main results}

In \S\ref{2} we recall some basic facts about critical values, along with more details on the correspondence between linear classes of quartic polynomials and elliptic curves with a $3$-isogeny. We then answer some natural questions left open after ACV I. Our previous proof of \eqref{h4} relied on the Cummins-Pauli classification for genus-$0$ subgroups of $\Gamma(1)$, but one would naturally seek a direct proof of this isomorphism. In \S\ref{3} we obtain such a proof by studying the geometry of the Galois closure of $f:\P^1_k\to\P^1_k$, which is a curve of genus $4$. In particular, we prove (cf. Proposition \ref{clos} and Theorem \ref{t1}):

\begin{teo}\label{t1'}
    Let $K$ be a number field and let $f\in K[X]$ be a quartic polynomial whose critical points are distinct and do not form an arithmetic progression. The Galois closure of the cover $f:\P^1\to\P^1$ factors through a $3$-isogeny of elliptic curves over $K$.
\end{teo}

In \S\ref{4} we investigate the global arithmetic of various families of curves that arise in the Galois diagram of Theorem \ref{t1'}---including that of \textit{critical elliptic curves} $E_f$, another elliptic family $E'_f$ admitting a $4$-congruence $E'_f[4]\simeq E_f[4]$, and a genus-$2$ family whose Jacobian is a glueing of $E_f$ and $E'_f$ along the $4$-congruence. We show that these families are fibrations over a suitably punctured projective line, and that the \textit{Prym} family $E'_f$ has a nontorsion section defined over $\Q(i)$. The latter will be central in the many of the applications of the later sections.

\S\ref{5} is devoted to the study of the local arithmetic of the critical families. We show that, for an infinite family of quartics $f\in\Q[X]$, the $4$-congruence found in \S\ref{4} induces a \textit{Selmer companionship}, a notion developed by Mazur and Rubin \cite{MR}. In particular, we construct the first infinite family of non-isogenous $2$-Selmer companions (cf$.$ Theorem \ref{companions}):
\begin{teo}\label{infcomp}
     There exist infinitely many pairs $(E_1,E_2)$ of non-isogenous rational elliptic curves which are $2$-Selmer companions over every number field, and such that no two pairs are related by quadratic twisting.
\end{teo}
For infinitely many of these pairs, we are able to extend the $2$-Selmer companionship to an isomorphism of $4$-Selmer groups (Theorem \ref{4sel}). Together with the nontorsion section on $E'_f$, this allows us to exhibit many rational elliptic curves with an element of order $4$ in their Tate-Shafarevich group, and to prove visibility in an abelian surface \cite{CM} for such elements (Proposition \ref{visibility}). Infinitude for this set of curves is conditional on the infinitude of even root number curves in a certain family.

In \S\ref{6} we highlight a mysterious connection between the Selmer companions of \S\ref{5} and the \textit{equicritical quartics} of \cite{nac}, before turning to global Diophantine applications. We settle, over a fixed quadratic extension, the infinitude question for rational solutions to the same-polynomial variables separated equation (cf$.$ Theorem \ref{t2}(ii)):
\begin{teo}\label{varsepbad}
    Let $f$ be a quartic with rational coefficients. The equation \begin{align*}
            f(X)=f(Y), \ X\neq Y
        \end{align*}has infinitely many solutions over $\Q(i)$, unless the $j$-invariant of the critical points of $f$ is $0,-27648/11$ or $55296/5$.
\end{teo}We also construct three rational genus-$4$ curves whose Jacobian splits over $\Q$ as the fourth power of an elliptic curve; two are novel, while the other recovers Bring's curve, the only previously known such example:
\begin{teo}\label{fourthsplit}
    For $f$ equal to each of the polynomials $$X^4+30X^2+64X, \quad X^4+90X^2+72X, \quad X^4+10X^2+40X,$$ the Jacobian of the Galois closure of $f:\P^1\to\P^1$ is isogenous over $\Q$ to the fourth power of an elliptic curve. The latter can be taken as having Weierstrass equation $$y^2=x^3-75x-506, \quad y^2=x^3+45x-18, \quad y^2=x^3-675x-79650,$$ respectively.
\end{teo}
\subsection{Acknowledgments and funding}
I am grateful to Jordan Ellenberg for sharing with me his insights on the topic of critical values, and for his suggestion of comparing the bad reduction of the critical and Prym families. I would also like to thank my advisors, Emmanuel Kowalski and Umberto Zannier, for their critical research advice.

\noindent
The author was partially supported by the Swiss National Science Foundation (grant number 219220).
\section{Critical values of quartic polynomials}\label{2}
In order to make this work self-contained, we briefly recall some basic facts about polynomials and their critical values. We kindly point the reader interested in more detail to \cite{nac}.

\subsection{Basics on critical values}

Recall that, for a polynomial $f\in k[X]$ of degree $d$, the critical values $C_f$ are the branch points of the induced map $f:\P^1_{k}\to\P^1_{k}$. They consist of $\infty$---over which there is total ramification---and the points in $f(Z_f)$, where the $Z_f$ is the (multi)set of roots of $f'$, i.e$.$ the finite ramification points. If $X_0\in Z_f$, we call $X_0$ a \textit{critical point} of $f$. Let us remark that the set of monic polynomials $f\in k[X]$ of degree $d$ and such that $f(0)=0$ can be identified with $\A^{d-1}_k$, via the bijection $f\mapsto Z_f$. The group of $k$-linear maps $\mathrm{Aff}(k)=\{z\mapsto az+b, \ (a,b)\in k^\times\times k\}$ acts coordinate-wise on $\A^{d-1}_k$, and we have\begin{align}
Z_{\alpha\circ f}=Z_f, \quad C_{\alpha\circ f}=\alpha(C_f)\label{inv1}
\end{align}for any $\alpha\in\mathrm{Aff}(k)$; so, restricting to monic polynomials with zero constant term loses no information about critical points and values. In the (generic, under the above identification) locus where critical points are distinct, a simple application of the Riemann--Hurwitz formula yields $$e_{X_0}(f)=2$$for any $x\in Z_f$. Moreover, as shown in \cite[Lemma 2.2]{nac}, critical values are also generically distinct; following \cite[4.4]{serTopGal}, we refer to a polynomial with distinct critical values as \textit{Morse.}

In addition to \eqref{inv1} we also have \begin{align}
    Z_{f\circ\alpha}=\alpha^{-1}(Z_f)\text{   and   }C_{f\circ\alpha}=C_f\label{inv2}
\end{align}for any $\alpha\in\mathrm{Aff}(k)$.
The map sending the $\mathrm{Aff}(k)$-class of $\hat x\in\A^{d-1}_k$ to that of $\hat y\in\A^{d-1}_k$ such that there is $f\in k[X]$ with $Z_f=\hat x$ and $C_f=\hat y$ is, therefore, well-defined. It is a generically étale morphism of algebraic varieties defined over $\Q$ \cite[\S2]{nac}. When $d=4$, it is a map between two curves, since the quotient by the affine relation reduces the dimension by $2$. We now let $k\subset\C$. The base curve is naturally $Y(1)$, with coordinate the $j$-invariant $j_{\text{CV}}$ of the critical values. In \cite{nac} we showed that the covering curve---the (reduced) Hurwitz space of quartic polynomials---is then simply the modular curve $Y_0(3)$ (or $Y_1(3)$---they are the same curve over $\Q$, see \cite[Corollary 4.2.2]{nac}) with coordinate the $j$-invariant $j$ of the quadruple of critical points (which here are taken to include $\infty$, see \cite[\S 4.2]{nac}). The above morphism is then given by \begin{align}
    j_{\text{CV}}=2^{-18}\frac{j(j-1536)^3}{j-1728}=\pi_3\left(\frac{j}{64}-27\right),\label{jcv}
\end{align}where $\pi_3:X_0(3)\longrightarrow X(1)$ is the classical Hauptmodul map $u\mapsto\frac{(u+3)^3(u+27)}{u}$ \cite[p.23]{suzy}. This is shown in \cite[Lemma 4.3]{nac}.
\subsection{The case of quartics}
\begin{defi}
    We call a quartic $f\in k[X]$ \textbf{good} if $j(Z_f)\notin\{1728,\infty\}$, otherwise we call $f$ \textbf{bad}.
\end{defi}
\begin{reml}
\begin{enumerate}[(i)]
    \item By \eqref{jcv}, a quartic is good if and only if the $j$-invariant of its critical values lies in $k$ (that is, it is not $\infty$).
    \item The bad quartics with $j(Z_f)=1728$ are those linearly equivalent (over $k$) to the Chebyshev quartic $8X^4-8X^2+1$, while those with $j(Z_f)=\infty$ are those linearly equivalent to $X^4+X^3$ or $X^4$. This follows, for instance, from the computations at the end of \cite[\S 5.2]{nac}.
\end{enumerate} \label{bad quartics}
\end{reml}
\begin{defi}
    Given a quartic $f\in K[X]$, the \textbf{critical curve relative to $f$} is the projective curve over $K$ with affine model\begin{align}
        E_f: y^2=\operatorname{Disc}_X(f(X)-x).\label{disc}
    \end{align}
\end{defi}
\begin{lem}
    The curve $E_f$ is smooth if and only if $f$ is good, in which case it is an elliptic curve and we have $E_f[2]=C_f$. If $f$ is monic, a Weierstrass model for $E_f$ over $K$ is given by \begin{align}
        -y^2=\prod_{v\in C_f}(x-v).\label{model}
    \end{align}\label{crit}
\end{lem}
\begin{proof}
    If $f$ has leading coefficient $A$, $\mathrm{Disc}_X(f(X)-x)$ is the cubic polynomial in $x$ with leading coefficient $-2^8A^3$ which vanishes at those $x_0\in k$ for which the polynomial $f(X)-x_0$ has a repeated root, i.e$.$ the critical values of $f$. The claim then follows from Remark \ref{bad quartics}(i) after sending $y\mapsto16y$.
\end{proof}
\begin{reml}\begin{enumerate}
    \item If $f$ is bad, a direct computation with the representatives in Remark \ref{bad quartics} (ii) shows that $E_f$ is a nodal cubic, except for the case of the linear equivalence class of $f(t)=t^4$.
    \item If $f$ is good we have $j(E_f)=j_{\text{CV}}=\pi_3\left(\frac{j(Z_f)}{64}-27\right)$, in virtue of Lemma \ref{crit} and \eqref{jcv}.
\end{enumerate}
\end{reml}
In the remainder of this work, we normalize our quartics $f\in K[X]$ to be monic and have no cubic nor constant terms: in formulas, $f(X)=X^4+2AX^2+4BX$. This can be achieved with a $K$-linear change of variable followed by a $K$-linear rescaling of the polynomial, so it has no impact on the equivalence classes of critical points and values, by \eqref{inv1} and \eqref{inv2}.
Thanks to the interpretation of $Y_0(3)$ as the moduli space of elliptic curves with a $3$-isogeny, our correspondence can be seen in the following way: 
\begin{prop}\label{3iso}
    Let $E/K$ be an elliptic curve with $j=j(E)\neq1728$, and let $Y^2=q(X)$ be a short Weierstrass equation for $E$. Define the quartic $f\in K[X]$ by\begin{align}
    f(X)=f_E(X):=\int_0^Xq(u)du.\label{3div}
\end{align} Then, the critical elliptic curve $E_f$ has a $3$-isogeny over $K$.
\end{prop} 

\begin{reml}\label{3pt}
    One sees, using for instance \cite[(19)-(20)]{nac}, that $-\prod_{v\in C_f}(3v-A^2)\in-\Q^2,$ so the twist by $-3$ of $E_f$ has a rational $3$-torsion point.
\end{reml}

\subsection{Some natural questions}\label{questions}

Picking up from \cite[Remark 2.1]{nac}, we now pose and answer some questions about critical values in degree $4$. For instance, it can be natural to ask which good quartics $f\in K[X]$ have one (resp$.$ three) $K$-rational critical value(s). Looking at the $G_K$-action on the equations \cite[(4)]{nac} defining them, one immediately sees that this is equivalent to $f$ having one (resp. three) $K$-rational critical point(s), which it maps to the $K$-rational critical value(s). Therefore, outside of the elliptic images $j=0,\,1728,\,\infty$ (where rationality of the moduli point can be induced by the extra automorphisms), such $f$ are simply those with $j(Z_f)$ in the image of the modular cover $X_1(2)(K)\to X(1)(K)$ (resp. $X(2)(K)\to X(1)(K)$).

In virtue of the classification and applications of equicritical quartics obtained in \cite{nac}, it may be of some interest to consider the same question while requiring our $f$ with one (resp$.$ three) critical value(s) in $K$ to also be equicritical to a quartic $g\in K[X]$ not linearly equivalent to it (an ``equicritical twin"). From \cite[\S 5]{nac} we know that quartics with an equicritical twin are rationally parametrized (up to acting linearly on the variable and the polynomial) by $$Y_3:=Y_0(3)\underset{Y(1)}{\times} Y_0(3)\setminus\Delta,$$ where $\Delta$ is the diagonal in the fiber product. Therefore, all such $f$ can be generated from a set of representative quartics with $j(Z_f)\in Y(K)$ by acting linearly on such quartics and the variable, where $Y:=Y_3\times Y_1(2)$ (resp. $Y:=Y_3\times Y(2)$) and the fiber product is over $Y(1)$. In the first case, the genus of the resulting cover of the $j_{\text{CV}}$-line $Y(1)$ is $0$ and hence, as $x^4-x$ is such a quartic (see \cite[Theorem 1]{nac}), there are infinitely many such $f$ over any number field $K$. More interesting is the case in which we require all critical values to be in $K$:

\begin{prop}
    Let $K$ be a number field. The following are equivalent:\begin{itemize}
        \item there exist infinitely many linearly inequivalent quartics $f\in K[X]$ with three $K$-rational critical values and which admit an equicritical twin in $K[X]$;
        \item the elliptic curve $E_0:y^2=x^3+1$ has positive $K$-rank.
        \end{itemize}
\end{prop}
\begin{proof}
    From the observations in the previous paragraph we find that a maximal set of such good $f$ bijects with that of $K$-rational points of $Y_3\times Y(2)\simeq Y_6$, where we again adopt Rubin and Silverberg's notation from \cite{rs}. In \cite[Theorem 2.1]{rs99} they show that $X_6$ has affine model $y^2=x^3+1$; since bad quartics account for finitely many cases up to equivalence, we are done by the Mordell-Weil Theorem.
\end{proof}
\begin{reml}
    We have $E_0(\Q)\simeq\Z/6\Z$, so there are finitely such $f$ with rational coefficients. Using the parametrization \cite[Theorem 1]{nac}, any such good $f$ would have to either share the same critical values of $x^4-x$---which are not all rational---or have, up to a linear change of variable, derivative $x^3-3tx-2$ for some $t\in\Q\setminus\{-2,0,1\}$. It is not hard to see that the only such cubic to have three rational roots is $x^3-3x-2=(x+1)^2(x-2)$, so there are no such good $f\in\Q[x]$, and all six rational points of $E_0$ correspond to bad quartics or cusps introduced by compactifying the affine model of $Y_6$.
\end{reml}

Another question we can ask is which quartics $f\in\overline\Q[X]$ satisfy \begin{align}
    Z_f=C_f.\label{cp=cv}
\end{align}This has, for instance, the application of exhibiting \textit{post-critically finite (PCF)} \cite{poirier} quartic polynomials, whose complete classification is still an open problem. We restrict to good quartics as the nontrivial case. Surely, equality of critical points and values implies equality of their $j$-invariants, so from \cite[(21)]{nac} we obtain the equation $$j=\frac{j(j-1536)^3}{2^{18}(j-1728)}$$for the $j$-invariant of the critical points. This means that either $j\in\{0,\infty\}$ or $u=j-1536$ satisfies $u^3=2^{18}(u-192)$; observing that $u=2^8=256$ is a solution, we find that such an $f$ must have \begin{align}
    j(Z_f)\in\{0,1792,128(11\pm\sqrt{13})\}.\label{cf=zf}
\end{align}
Now we have to look for $f$ satisfying \eqref{cp=cv} inside each of the four linear equivalence classes satisfying \eqref{cf=zf}. Take any such $f$ with $j(Z_f)=j=j(C_f)$; then there exists $\alpha\in\PGL_2(\C)$ mapping $Z_f\cup\{\infty\}$ to $C_f\cup\{\infty\}$ as sets. Up to post-composing $\alpha$ with the element $\beta\in\PGL_2(\C)$ that fixes the set $C_f\cup\{\infty\}$ and swaps $\infty$ with $\alpha(\infty)$, we can assume that $\alpha$ is affine. Letting $f_1=\alpha^{-1}\circ f$ and $f_2=f\circ\alpha^{-1}$, we then have $Z_{f_i}=C_{f_i}$ in virtue of \eqref{inv1} and \eqref{inv2}. Moreover, all quartics linearly equivalent to $f$ satisfying \eqref{cp=cv} can clearly be obtained from this procedure for such an $\alpha$, so their set is $$\bigcup_{\alpha\in\mathrm{Aff}(\C)}\alpha\mathrm{Stab}_{\mathrm{Aff}(\C)}(Z_f)\circ f_1\circ\mathrm{Stab}_{\mathrm{Aff}(\C)}(Z_f)\alpha^{-1}.$$Now, for $j\neq0,1728$ those stabilizers are trivial, as a quadruple with such a $j$-invariant is stabilized exactly by a Klein-four subgroup $C_2\times C_2\leq\PGL_2(\C)$ acting faithfully on it. Therefore, for $j(Z_f)\in\{1792,128(11\pm\sqrt{13})\}$ there is exactly one such quartic up to the conjugation action of $\mathrm{Aff}(\C)$. We can find a representative with minimal coefficient field as follows: take any elliptic curve $E_j/\Q(j)$ with $j$-invariant $j$ and let $X_i$ be its Weierstrass roots and $f$ a primitive of its Weierstrass polynomial. Compute $x_i=f(X_i)$ and take an affine transformation $\alpha\in\mathrm{Aff}(\Q(j))$ mapping $\{X_i\}$ to $\{x_i\}$ (this exists by construction). Then, $\alpha^{-1}\circ f\in\Q(j)[X]$ is such a representative. Taking $j=1792$ and $f(X)=X^4-42X^2-28X-154$---whose critical points are the Weierstrass roots of the elliptic curve with Cremona label 1764j1, having $j$-invariant $1792$---we find that $C_f=63Z_f-448$. This yields an example of PCF quartic conjugacy class which is, to the best of our knowledge, new (see \cite[Theorem 4.9]{fraser}):
\begin{cor}
    The quartic $\frac1{63}(X^4-42X^2-28X+294)$ is post-critically finite.
\end{cor}


For $j=0,$ every quartic $f\in K[X]$ is linearly equivalent over $K$ to $f_B(x)=X^4+4BX$ for some $B\in K$, and equivalence classes correspond to those of $B$ in $K^\times/(K^\times)^3$. We have $Z_{f_B}=\{\zeta,\zeta\omega,\zeta\omega^2\}$ where $\zeta^3=-B$ and $\omega$ is a primitive third root of unity. On the other hand, $C_{f_B}=\{3B\zeta\omega^i, \ i=0,1,2\}$. Therefore, \eqref{cf=zf} is satisfied by the polynomials $\alpha^{-1}\circ f_B$ and $f_B\circ \alpha^{-1}$ for $\alpha(t)=3B\omega^it$ and $i=0,1,2$: the cases where $i=1,2$ correspond to the extra nontrivial affine automorphism of a quadruple with $j$-invariant $0$. Focusing on rational quartics, we have showed:

\begin{prop}
    Up to the conjugacy action of $\mathrm{Aff}(\Q),$ the set of quartics $f\in\Q[X]$ with distinct critical points such that $Z_f=C_f$ has the following representatives:\begin{itemize}
        \item $\frac1{63}(X^4-42X^2-28X+294),$
        \item $\frac1{3B}(X^4+4BX)$ for $B\in\Q^\times/(\Q^\times)^3$.
    \end{itemize}
\end{prop}
\begin{reml}
    While for $f_B$ of the second type we clearly have $f_B(X_i)=X_i$, so the critical points are preserved point-wise, it is easy to see that for $f(X)=\frac1{63}(X^4-42X^2-28X+294)$ we have $f(X_i)=X_{\sigma(i)}$ for $\sigma$ a $3$-cycle.
\end{reml}
Finally, in the spirit of \cite[Theorem 2]{nac}, we can ask \textit{how likely} it is for a Galois orbit $\{x_1,x_2,x_3\}$ defined over $K$ to be \textit{critical over $K$}, that is, to consist of the finite critical values of a quartic in $K[X]$. In order to make such a question precise, we need to put a height function on the set of such triples; the condition that $G_K$ acts on each of them makes this equivalent to giving a height function on the set of cubics with coefficients in $K$. Let $\mathrm{H}:K^\times\to[1,\infty)$ be the usual multiplicative Weil height \cite[\S 1.2]{zannierhts} and, following Schanuel \cite{schaunel}, we work with the relative height $\mathrm{Ht}=H^{n}$, where $[K:\Q]=n$. In order to uniformize for scaling, the natural choice for our height is the following:
\begin{defi}\label{height}
    Let $\mathcal{S}=\{X_1,X_2,X_3\}\subset\overline\Q$ satisfy $(X-X_1)(X-X_2)(X-X_3)=X^3+aX^2+bX+c:=p(X)\in K[X]$. We set $$\mathrm{Ht}(\mathcal{S})=\mathrm{Ht}(p):=\max(\mathrm{Ht}(a)^6,\mathrm{Ht}(b)^3,\mathrm{Ht}(c)^2)$$and$$T_K(H)=\{\mathcal{S}=\{X_1,X_2,X_3\}\subset\overline\Q \ G_K\text{-stable}:\mathrm{Ht}(\mathcal{S})\le H\}.$$
\end{defi}
$G_K$-stable triples and cubics over $K$ with nonzero discriminant biject, so $|T_K(H)|\asymp H^{\frac26+\frac23+1}=H^2$, as the vanishing discriminant cuts out a Zariski closed subset of codimension $1$. Therefore, the logarithmic density of $K$-critical triples with respect to our height is \begin{align}\label{density}
    \delta_{\text{crit}}(K):=\liminf_{H\to\infty}\frac{\log\left(\#\{\mathcal{S}=\{X_1,X_2,X_3\}\subset\overline\Q:\exists f\in K[X]:C_f=\mathcal{S}, \ \mathrm{Ht}(\mathcal{S})\le H\}\right)}{2\log H}.
\end{align}
We have:
\begin{prop}
    For a number field $K$, we have $$\delta_{\text{crit}}(K)\ge\frac23.$$
\end{prop}
\begin{proof}
    In virtue of Definition \ref{height}, we can analogously count (not necessarily short) Weierstrass models (and not isomorphism classes) of elliptic curves over $K$ with height of the model at most $H$.
    We focus just on those with $j$-invariant $0$, which is a critical $j$-invariant over $\Q$. Since the substitution $x\to x-\frac{a}3$ brings the model in short Weierstrass form, any such model has the shape $$y^2=p(x)=x^3+ax^2+\frac{a^2}3x+c, \quad c\neq\frac{a^3}{27}.$$ The condition $\mathrm{Ht}(p)\le H$ is then equivalent to having both $\mathrm{Ht}(a)\le H^{\frac16}$ and $\mathrm{Ht}(c)\le H^{\frac12}$. This gives $\asymp H^{\frac13+1}$ choices by Schanuel's Theorem, so the numerator in \eqref{density} is at least $\frac43\log H$.
\end{proof}
\begin{reml}
    The lower bound of $\frac23$ for $\delta_{\mathrm{crit}}(K)$ is likely to be sharp. Indeed, for the similar (but subtly different) problem of counting $K$-isomorphism classes of elliptic curves over $K$ with a $3$-isogeny, the relative log-density for $K=\Q$ is $\frac35$, but it drops to $\frac{2}{5}$ if we remove the $j=0$ family, see \cite[Proposition 4.2]{bs}. One expects this family to dominate the count also in our case; however, showing this requires good quantitative control on the amount of cancellation that can happen when changing variables from a minimal model, which seems to be a tricky problem. Moreover, for $K\neq\Q$, the issue of the general lack of a global minimal model makes working with equations even harder.
\end{reml}
\section{Inside the Galois closure}\label{3}

\subsection{Genus $2$ covers of critical elliptic curves}\label{3.1}
The modular curve $Y_0(3)$ is already known to be the reduced Hurwitz space parametrizing classes of maps $\gamma:\P^1\to\P^1$ of degree $3$ ramifying simply above each of the four branch points \cite[\S 4.2]{fr1}. This is the case because these covers are precisely those induced by $3$-isogenies of elliptic curves $E_1\to E_2$, after quotienting by the respective involutions:
 \begin{center}
        \begin{tikzpicture}[>=stealth,->,node distance=3cm]
  \node (X) at (0,2) {$E$};
  \node (Y) at (3,2) {$E'$};
  \node (P1) at (0,0) {$\P^1$};
  \node (P2) at (3,0) {$\P^1$};

  \draw[->] (X) to node[above] {} (Y);
  \draw[->] (X) to node[left] {$\pi_E$} (P1);
  \draw[->] (Y) to node[right] {$\pi_{E'}$} (P2);
  \draw[->] (P1) to node[above] {$\gamma$} (P2);
\end{tikzpicture}
    \end{center}

One may therefore be tempted to make quartic polynomials take up a similar a role, by considering a larger space of maps to a target elliptic curve, and see if they can arise as the induced maps $\P^1_X\to\P^1_x$. This is indeed the case:
\begin{defi}\label{precrit}
    Let $f\in K[X]$ be a good quartic with critical points $X_1,X_2,X_3$ and let $E_f$ be the associated critical elliptic curve. For $i=1,2,3$, let $f^{-1}(f(X_i))\setminus\{X_i\}=\{X_{i1},X_{i2}\}$, and define 
    \begin{align}\label{qsplit}
        h_f(X)=\prod_{i=1}^3(X-X_{i1})(X-X_{i2})
    \end{align}to be the \textbf{precritical sextic relative to} $f$.
\end{defi}
\begin{reml}
    Since $f$ is good, $h_f$ has no repeated root. Therefore, the normalization of the curve $\c_f:-Y^2=h_f(X)$ is a smooth genus-$2$ curve with divisor at infinity supported at two points.
\end{reml}
\begin{prop}\label{cover}
    The curve $\c_f$ is defined over $K$. There is a map $\phi_f:\c_f\to E_f$ given by $(X,Y)\mapsto(f(X),\frac{1}{4}Yf'(X))$, for which the diagram
    \begin{center}
        \begin{tikzpicture}[>=stealth,->,node distance=3cm]
  \node (X) at (0,2) {$\c_f$};
  \node (Y) at (3,2) {$E_f$};
  \node (P1) at (0,0) {$\P^1_X$};
  \node (P2) at (3,0) {$\P^1_x$};

  \draw[->] (X) to node[above] {$\phi_f$} (Y);
  \draw[->] (X) to node[left] {$\pi_{\c_f}$} (P1);
  \draw[->] (Y) to node[right] {$\pi_{E_f}$} (P2);
  \draw[->] (P1) to node[above] {$f$} (P2);
\end{tikzpicture}\label{diag}
    \end{center}
    commutes.
\end{prop}
\begin{proof}
    If $\phi_f$ is well defined then the commutativity follows trivially, so we just need to verify that \begin{align*}y^2=\frac{1}{16}Y^2(f'(X))^2=-\frac{1}{16}(f'(X))^2h_f(X)\overset{?}{=}-p(f(X)),\end{align*} where $p$ is the polynomial at the right-hand side of \eqref{model}. Since $f$ has distinct critical values $f(X_i)$, we have \begin{align*}
        p(f(X))=\prod_{i=1}^3(f(X)-f(X_i))=\prod_{i=1}^3(X-X_i)^2(X-X_{i1})(X-X_{i2})=\frac{1}{16}(f'(X))^2h_f(X).
    \end{align*}
    As $f$ and $p$ have coefficients in $K$, the last equation implies $h_f\in K[X],$ so $\c_f$ is defined over $K$.
\end{proof}
\begin{reml}\label{inf2}
    As $\pi_{\c_f}$ and $\pi_{E_f}$ have the same degree, $\c_f$ is just the normalization of the fiber product $\P^1_X\underset{\P^1_x}{\times} E_f$, the map $\P^1_X\to \P^1_x$ being $f$. Both points at infinity $\infty^{\pm}$ are sent to $0_{E_f}$ by $\phi_f$.
\end{reml}

Now the question becomes: what does the cover $\phi_f$ have to do with $E_f$ having a $3$-isogeny? As Fried remarks in \cite{fr1}, one-dimensional (reduced) Hurwitz spaces of covers of $\P^1$ are always quotients of the upper half plane $\h$ by finite index subgroups of $\Gamma(1)$---with the $j$-line $\Gamma(1)\backslash\h$ taken as parametrizing the branching locus; but they need not be \textit{congruence} modular curves.

In recent work \cite{chen}, Chen surveys how noncongruence modular curves can be viewed as (components of) moduli spaces $\m(H)$ for $H$-covers $\phi:C\to E$ of elliptic curves branched at most above the origin. He remarks that, while the \textit{abelian} case---that of unramified covers, and hence isogenies \cite[Proposition 3.4]{chen}---clearly yields congruence modular curves, the converse does not hold \cite[Proposition 3.4]{chen}. The reader can find the precise definition of the stacks $\m(H)$ in \cite[\S 3.2]{chen}; they are, in essence, reduced Hurwitz spaces for $H$-covers of a variable elliptic curve branching only above the origin. Here we use the letter $H$ to denote the Galois group of the cover of the (critical, in our case) elliptic curve, to differentiate from that $G$ of its projection; indeed, both in the dihedral case coming from isogenies and in our own, one has $G\simeq H\rtimes C_2.$ In Theorem \ref{t1}, we exhibit an instance of nonabelian Galois group $H$ such that $\m(H)$ is a congruence modular curve, by showing that $\m(A_4)\simeq Y_0(3)$ over $\Q$. 

\subsection{The Galois closure}
As anticipated, we now exploit our construction to directly show why one gets $Y_0(3)$ when parametrizing good quartic polynomials with given critical values. The core idea is to look at the Galois closure $X_f$ of our quartic $f$:
\begin{prop}\label{clos}
    \begin{enumerate}[(i)]
        \item There is a map $g_f:X_f\to\c_f$ over $K$, which makes $X_f$ into the Galois closure of $\phi_f$.
        \item There exists an elliptic curve $\widetilde E_f/K$ and $K$-rational maps $\nu_f, \psi_f$ of degrees $4$ and $3$ respectively, with $\psi_f$ unramified, such that the following diagram commutes:
        
        \begin{center}
        \begin{tikzpicture}[>=stealth,->,node distance=3cm]
  \node (X) at (0,2) {$\c_f$};
  \node (Y) at (3,2) {$E_f$};
  \node (P1) at (0,0) {$\P^1_X$};
  \node (P2) at (3,0) {$\P^1_x$};
  \node (P) at (-2.5,3.5) {$X_f$};
  \node (Q) at (0.5,3.5) {$\widetilde E_f$};

  \draw[->] (P) to node[above] {$\nu_f$} (Q);
  \draw[->] (Q) to node[above] {$\psi_f$} (Y);
  \draw[->] (X) to node[above] {$\phi_f$} (Y);
  \draw[->] (X) to node[left] {$\pi_{\c_f}$} (P1);
  \draw[->] (Y) to node[right] {$\pi_{E_f}$} (P2);
  \draw[->] (P1) to node[above] {$f$} (P2);
  \draw[->] (P) to node[below] {$g_f$} (X);
\end{tikzpicture}\label{diag2}
    \end{center}
    \end{enumerate}
\end{prop}
\begin{proof}
    \begin{enumerate}[(i)]
        \item We know that $f:\P^1_X\to\P^1_x$ has Galois group $G_f\simeq S_4\simeq A_4\rtimes C_2$, with the normal subgroup $A_4$ corresponding to the double cover of $\P^1_x$ obtained by adjoining to the function field $K(x)$ a root of the discriminant. This is simply $y^2=\mathrm{Disc}_X(f(X)-x)$, which is the critical elliptic curve $E_f$. Since $\c_f\simeq E_f\underset{\P^1_x}{\times}\P^1_X$ by Remark \ref{inf2}, the universal property of fiber products ensures that $X_f\to\P^1_x$ factors through $\c_f$. We know that $X_f\to E_f$ is Galois; if it had a Galois subcover $D\to E_f$ factoring through $\c_f$, then $A_4$ would have a normal subgroup contained in $C_3$, but it does not.
        \vspace{2mm}
        
        \item Since $A_4\simeq V_4\rtimes C_3$ with $V_4$ the Klein group $C_2\times C_2$, there is a $C_3$-Galois cover $$\psi_f:X_f/V_4\to E_f$$ branching at most over $E_f[2]$. If this cover branched above $P\in E_f(k)\setminus\{0_{E_f}\}$, there would be $Q\in X_f(k)$ with $e_{\pi_{E_f}(P)}(Q)>2$, contradicting the fact that the inertia is generated by the local monodromy, which is a transposition above $\pi_{E_f}(P)$. Being abelian and branched at most above the origin, $\psi_f$ is an isogeny from another elliptic curve $\widetilde E_f$ \cite[\S 1.2]{chen}. All objects are defined over $K$ since we are inside the Galois closure of $f$, so we are done.
    \end{enumerate}
\end{proof}
We now need the following converse to Propositions \ref{cover} and \ref{clos}(i):
\begin{lem}
\label{corr}
    Let $\phi:C\to E$ be a degree $4$ cover with Galois group $A_4$, where $(C,\pi_C)$ is a genus $2$ \textbf{real} hyperelliptic curve and $(E,0_E,\pi_E)$ is an elliptic curve in Weierstrass form. If $\phi^{-1}(0_E)=\{\infty^+_C, \ \infty^-_C\}$, the induced map $f_{\phi}:\P^1_t\to\P^1_x$ through $\pi_C$ and $\pi_E$ is a good quartic polynomial.
\end{lem}

\begin{proof}
    The map $f_{\phi}:\P^1_t\to\P^1_x$ totally ramifies above $\infty$ (and nowhere else, by the Riemann--Hurwitz formula) with only preimage $\infty$, so it is a quartic polynomial. It now suffices to show that that its Galois closure $X_{f_\phi}$ factors through $E$: in that case, by Remark \ref{inf2}, $X_{f_\phi}$ factors through the Galois closure $X$ of $\phi$; since this cover has degree $24$ over $\P^1_x$ in virtue of our assumption on $\Gal(X/E)$, the two Galois closures must coincide, as $[K(X_{f_\phi}):K(x)]\le|S_4|=24$. Therefore, $K(E)=K(X_{f_\phi})^{A_4}$, which means that $E$ is the discriminant curve of $f_\phi$; as $E$ is smooth, $f$ must then be good by Lemma \ref{crit}. So, let $X$ be the Galois closure of $\phi$. As $K(C)=K(E)(t),$ the function field of $X$ is the splitting field of $f_\phi(t)-x$ over $K(E),$ so $K(X)=K(X_{f_\phi})K(E)$. We know that $[K(X):K(E)]=12,$ so, if $K(E)\not\subset K(X_{f_\phi})$, then $[K(X_{f_\phi}):K(x)]=\frac12[K(X):K(x)]=12$. But $K(X_{f_\phi})/K(x)$ is Galois with Galois group $\triangleleft \ S_4$, so necessarily $A_4$. On the other hand, the Galois group contains the inertia group, so in particular the $4$-cycle given by local monodromy of $f_\phi$ at infinity. Since $A_4$ does not contain a $4$-cycle, we are done.
\end{proof}
We can finally show:
\begin{thm}\label{t1}
    We have $\H_4\simeq\m(A_4)\simeq Y_0(3)$ over $\Q$.
\end{thm}
\begin{proof}
 Observe that $\H_4$ is naturally a connected component of $\m(A_4)$ by Proposition \ref{clos}(i). We show that it is the only one: any $A_4$-cover $X\to E$ branched only above the origin factors through the $C_3$-cover $X/V_4\to E$, which is a $3$-isogeny as remarked in the proof of Proposition \ref{clos}(ii). The Galois cover $X\to X/V_4$ branches at most over the kernel of the $3$-isogeny, with all preimages having the same ramification index $e$: a priori, either $1$, $2$ or $4$. In the first case, $X$ would have genus $1$, making $X\to E$ an isogeny; this must have abelian monodromy, contradicting the fact that the Galois group is $A_4$. The third case is ruled out by the Riemann--Hurwitz formula: we would need to have $2\mid3(e-1)$ in the fully branched case. Therefore, $X$ has genus $4$ and $X\to E$ has the same ramification as that of $X_f\to E_f$: a double transposition above each element of the kernel of the $3$-isogeny. In particular, $0_{E}$ has six preimages in $X$, all with ramification index $2$.
 
 \noindent
 Set \begin{align*}l:X\to C=X/C_3, \quad \psi: C\to E_f.\end{align*} If $g(C)=1$, then $\psi$ would be unramified, so $0_{E}$ would have $4$ preimages in $C(k)$. But then it could not possibly have six in $X(k)$: in that case, at least four of them would have odd ramification index ($1$ or $3$). On the other hand, $$3(2g(C)-2)\le 2g(X)-2=6\Longrightarrow g(C)\le2$$ so $C$ has genus $2$ and $l$ is unramified. Therefore, $\psi^{-1}(0_{E})$ consists of two points. Let $i$ be the hyperelliptic involution of $C$. Observe that $\psi+\psi\circ i$ equals a constant (\cite[Remark 2.3]{GLNZ}), which must then be the branch locus of $\psi$, so $0_E$ in our case; in other words, we have $\psi\circ i=-\psi$. If the two preimages of $0_E$ were two Weierstrass points, necessarily each with ramification index $2$ over $0_E$, $\psi$ would be invariant under $i$. Hence $i$ must swap the two preimages of $0_E$.
 Choosing a model of $X$---and hence $C$---so that $\psi^{-1}(0_{E})=\{\infty^+_C, \ \infty^-_C\}$ (which by definition does not change the class in the reduced Hurwitz space), Lemma \ref{corr} gives us that $C\to E$ is the lift of a good quartic. We have shown the first isomorphism.

\noindent
Now, not only $A_4$-covers factor through a $3$-isogeny, but $\m(A_4)\to Y(1)$ factors through $\m(C_3)\simeq Y_0(3)$ by the standard theory of orbifold covers \cite[Proposition 3.17]{chen}. Since $\H_4\to Y(1)$ has degree $4$ \cite[Proposition 3.4]{nac}, all curves have degree $4$ over the $j$-line and are therefore isomorphic.
\end{proof}

Having exploited the structure of the Galois closure, it is natural to look at the ``completed" diagram---that is, at all the factorizations of $X_f\to\P^1_x$ through $X_f/K$, for some $K\leq S_4$. Of course, we can reason up to conjugation, easing the load on terminology. 

Recall the interpretation of the Galois closure of a degree $d$ cover as the normalization of any component of its $d$-fold self-fibered product with the \textit{fat diagonal} $\Delta$ removed \cite[8.3.2]{RET}. In our case, letting \begin{align}\label{selffp}
    C'_f=\P^1_X\underset{\P^1_x}{\times}\P^1_X\setminus\Delta_f, \quad E'_f=\widetilde{C'_f},
\end{align}$E'_f$ has a degree-$3$ cover to $\P^1_X$ and we have \begin{align*}X_f\simeq (E'_f\underset{\P^1_X}{\times}E'_f)\setminus\Delta_{E'_f}.\end{align*}Therefore, $E'_f\simeq X_f/\langle\tau\rangle$, with $\tau\in S_4$ any transposition. 
The subgroup generated by any double transposition $\langle\{\tau,\tau'\}\rangle$ defines a genus $2$ cover \begin{align*}Y_f=X_f/\langle\{\tau,\tau'\}\rangle\end{align*} of $\widetilde E_f$. $Y_f$ also has a map to $X_f/C_4$ (since the square of a $4$-cycle is a double transposition). Moreover, the extension corresponding to $D_4\le S_4$ is that of the cubic resolvent cover $z\mapsto\mathrm{Res}_f(z)$ associated to $f$.
\begin{reml}\label{resolvent}
    The $D_4$-extension is also that of the cover in the bottom row of the first diagram in \S\ref{3.1} for the $3$-isogeny $\psi_f$. Therefore, cubic resolvents of quartic polynomial covers over $K$ parametrize $3$-isogenies of elliptic curves over $K$.
\end{reml}

We summarize all this, along with facts we previously worked out, in the following partial table of subgroups of $S_4$ up to conjugacy, along with their respective quotients of $X_f$:
\newpage
\begin{center}
\setlength{\abovecaptionskip}{0.5cm}
\setlength{\belowcaptionskip}{0.5cm}
\captionof{table}{Subgroups of $S_4$ up to conjugacy and relative quotients of $X_f$.}
\begin{tabular}{|l|l|l|c|l|}
\hline
\textbf{Subgroup} & \textbf{Generators} & \textbf{Quotient} & \textbf{Genus} & \textbf{Jacobian isogeny class} \\
\hline
$\{e\}$ 
  & $e$ 
  & $X_f$ & $4$ & ? \\
\hline

$C_2^{\text{t}}$ (transposition)
  & $\langle (12)\rangle$ 
  & $E'_f$ & ? & $?$ \\
\hline

$C_2^{\text{d}}$ (double transposition)
  & $\langle (12)(34)\rangle$ 
  & $Y_f$ & $2$ & ? \\
\hline

$C_3$
  & $\langle (123)\rangle$ 
  & $\c_f$ & $2$ & ? \\
\hline

$C_4$
  & $\langle (1234)\rangle$ 
  & $?$ & $?$ & ? \\
\hline

$C_2\times C_2$ (non-normal Klein) 
  & $\langle (12),(34)\rangle$ 
  & $?$ & $?$ & ? \\
\hline

$V_4\simeq C_2\times C_2$ (normal Klein)
  & $\langle (12)(34),(13)(24)\rangle$ 
  & $\widetilde E_f$ & $1$ & $E_f$ \\
\hline

$S_3$
  & $\langle (12),(123)\rangle$ 
  & $\P^1_X$ & $0$ & trivial \\
\hline

$D_4$
  & $\langle (1234),(13)\rangle$ 
  & $\P^1$ & $0$ & trivial \\
\hline

$A_4$
  & $\langle (123),(12)(34)\rangle$ 
  & $E_f$ & $1$ & $E_f$ \\
\hline

$S_4$
  & $\langle (12),(1234)\rangle$ 
  & $\P^1_x$ & $0$ & trivial \\
\hline
\end{tabular}
\vspace{4mm}

\label{tab1}
    
\end{center}

In the next section we will use the theory of \textit{Brauer relations} to complete the table and recover information about the Jacobians of the curves that appear in it.

\section{Global arithmetic of the critical families}\label{4}
\subsection{Isogenies and Brauer relations}

A \textit{Brauer relation} for a finite group $G$ is a formal finite linear combination 
of transitive $G$-sets\begin{align}\label{br}
    \sum_i a_i [G/H_i] - \sum_j b_j [G/K_j], \quad a_i,\,b_j>0
\end{align}in the Burnside ring $B(G)$---that is, an isomorphism$$\bigoplus_i \mathbb{Q}[G/H_i]^{a_i} \cong \bigoplus_j \mathbb{Q}[G/K_j]^{b_j}$$of $\Q[G]$-modules between the induced
permutation representations.

Kani--Rosen theory \cite{kanirosen} tells us that, if $\Aut_K(X)\simeq G$, then \eqref{br} induces an isogeny $$\bigoplus_i\jac(X/H_i)^{a_i}\sim\bigoplus_j\jac(X/K_j)^{b_j}$$defined over $K$ \cite[Theorem 1.5]{DGKM}. Let $B(S_4)$ be the vector space of Brauer relations for $S_4$ and let $\Theta=(\Theta_1,...,\Theta_6)$ be its basis described in \cite[Proposition 3.1]{waka}; setting $R=\Theta M$ for the matrix 
$$M=\begin{pmatrix}
1&3&1&3&0&0\\
2&0&0&0&0&0\\
-1&-1&-1&-1&0&0\\
0&0&0&2&-1&0\\
0&-2&-1&-2&-1&0\\
0&0&1&0&0&-1
\end{pmatrix}$$
of determinant $8$, we find that the following six elements form a basis for $B(S_4)$:
\begin{align*}
&R_1: [A_4] + 2[D_4] = 2[S_4] + [V_4] \\
&R_2: 3[A_4] + 2[C_2^t] = [V_4] + 2[S_3] + 2[A_3] \\
&R_3: [A_4] + [C_2^d] = [V_4] + [A_3] \\
&R_4: 3[A_4] + 2[S_3] + 2[C_4] = 4[S_4] + [V_4] + 2[A_3] \\
&R_5: [\{e\}] + 2[S_4] = [C_2^t] + [C_3] + [C_4] \\
&R_6: [\{e\}] + 2[C_2^2] = 2[C_2^{\text{t}}] + [C_2^{\text{d}}] 
\end{align*}
\noindent
$R_1$ gives $E_f\sim\widetilde E_f$, which we already knew; 
$R_2$ gives \begin{align}\label{jacc}
    E^3_f\times (\jac(E'_f))^2\sim \widetilde E_f\times\jac(\c_f)^2\Longrightarrow\jac(\c_f)\sim E_f\times \jac(E'_f),
\end{align}
so, in particular, $E'_f$ has genus $1$.
$R_3$ gives \begin{align*}
    E_f\times\jac(Y_f)\sim\widetilde E_f\times\jac(\c_f),
\end{align*}so $\jac(Y_f)\sim\jac(\c_f)\sim E_f\times \jac(E'_f)$. $R_4$ gives \begin{align*}
    E_f^3\times\jac(X_f/C_4)^2\sim \widetilde E_f\times\jac(\c_f)^2,
\end{align*}so $\jac(X_f/C_4)\sim \jac(E'_f)$. $R_5$ then gives \begin{align}
    \jac(X_f)\sim E_f\times\jac(E'_f)^3\label{g4decompose}
\end{align}and finally $R_6$ gives $g(X_f/C_2^2)=0$. This allows us to complete our table 1:

\begin{center}
\setlength{\abovecaptionskip}{0.5cm}
\setlength{\belowcaptionskip}{0.5cm}
\captionof{table}{Subgroups of $S_4$ up to conjugacy and relative quotients of $X_f$.}
\begin{tabular}{|l|l|l|c|l|}
\hline
\textbf{Subgroup} & \textbf{Generators} & \textbf{Quotient} & \textbf{Genus} & \textbf{Jacobian isogeny class} \\
\hline
$\{e\}$ 
  & $e$ 
  & $X_f$ & $4$ & $E_f\times \jac(E'_f)^3$ \\
\hline

$C_2^{\text{t}}$ (transposition)
  & $\langle (12)\rangle$ 
  & $E'_f$ & $1$ & $\jac(E'_f)$ \\
\hline

$C_2^{\text{d}}$ (double transposition)
  & $\langle (12)(34)\rangle$ 
  & $Y_f$ & $2$ & $E_f\times \jac(E'_f)$ \\
\hline

$C_3$
  & $\langle (123)\rangle$ 
  & $\c_f$ & $2$ & $E_f\times \jac(E'_f)$\\
\hline

$C_4$
  & $\langle (1234)\rangle$ 
  & $X_f/C_4$ & $1$ & $\jac(E'_f)$ \\
\hline

$C_2\times C_2$ (non-normal Klein) 
  & $\langle (12),(34)\rangle$ 
  & $\P^1$ & $0$ & trivial \\
\hline

$V_4\simeq C_2\times C_2$ (normal Klein) 
  & $\langle (12)(34),(13)(24)\rangle$ 
  & $\widetilde E_f$ & $1$ & $E_f$ \\
\hline

$S_3$
  & $\langle (12),(123)\rangle$ 
  & $\P^1_X$ & $0$ & trivial \\
\hline

$D_4$
  & $\langle (1234),(13)\rangle$ 
  & $\P^1$ & $0$ & trivial \\
\hline

$A_4$
  & $\langle (123),(12)(34)\rangle$ 
  & $E_f$ & $1$ & $E_f$ \\
\hline

$S_4$
  & $\langle (12),(1234)\rangle$ 
  & $\P^1_x$ & $0$ & trivial \\
\hline
\end{tabular}
\vspace{4mm}

\label{tab2}
    
\end{center}

In particular, this implies that $\jac(X_f)$ is isogenous to the product of $\jac(\c_f)$ with the abelian surface $(E'_f)^2$, and we have \begin{align*}
   L(\jac(X_f),s)=L(E_f,s)L(E'_f,s)^3=L(\jac(\c_f),s)L(E'_f,s)^2.
\end{align*}

Let us look at the isogeny $\jac(\c_f)\sim E_f\times \jac(E'_f)$ in more detail. It is customary to refer to a cover $\phi:C\to E$ of an elliptic curve as \textit{optimal} if, for any factorization $C\to E_1\to E$ of $\phi$ with $E_1$ an elliptic curve, $E_1\to E$ is an isomorphism. It is well-known that, in characteristic zero, an optimal genus-$2$ cover $C\to E$ of an elliptic curve induces a complementary (optimal) cover $C\to E'$ of the same degree, call it $d$, and a splitting $\jac(C)\to E\times E'$ which is a $(d,d)-$isogeny, see \cite{BD}.

If $f$ is a good quartic, then $E_f$ is an elliptic curve by Lemma \ref{crit}. $A_4$ has no index $2$ subgroup and $\deg\phi_f=4$, so $\phi_f$ is optimal; it therefore induces a degree-$4$ cover from $\c_f$ to another elliptic curve $D_f$ over $K$, such that \begin{align}\label{prym}
    0\to D_f\to\jac(\c_f)\to E_f\to 0
\end{align}is exact \cite[\S 2, Lemma]{kuhn}. We then have that $E_f\times D_f$ is $(4,4)-$isogenous over $K$ to $\jac(\c_f)$, so $D_f$ is isogenous to $\jac(E'_f)$ by \eqref{jacc}. In \cite[Lemma 3.1]{GLNZ}, together with Gallese, Lombardo and Zannier we show how $D_f$ is isomorphic over $K$ to the curve $E'_f$ of \eqref{selffp}.
In light of this, we will from now on refer to the uniquely determined Prym elliptic curve satisfying \eqref{prym} as $E'_f$. \eqref{prym} then gives:
\begin{lem}\label{4-iso}
    There is an anti-isometry between $E_f[4]$ and $E'_f[4]$, i.e$.$ a $G_{K}$-equivariant isomorphism $\iota:E_f[4]\simeq E'_f[4]$ that inverts the Weil pairing. 
\end{lem}

In \cite{BD}, Bruin and Doerksen study genus $2$ curves with $(4,4)$-split Jacobians. In Section 4 they show how such a curve admits a model $Y^2=F_1(X)F_2(X)F_3(X)$ with $\{F_1, F_2, F_3\}$ Galois-stable, a so-called quadratic splitting. Notice how in our case such quadratic splitting is given (up to sign) by the factors in \eqref{qsplit}. If the quadratic splitting is \textit{nonsingular}, they show that the $(4,4)$-isogeny $E_1\times E_2\to\jac(C)$ must factor through a $(2,2)$-isogeny to the Jacobian of a genus-$2$ curve; indeed, in our case one easily sees that this can be taken as $Y_f$, via its $2$-covers to both $X_f/C_4$ and $X_f/V_4$ in the previous table. In \cite{kos}, Konstantinou defines \textit{pseudo Brauer relations} and shows that many classes of well-known isogenies are induced by a generalization of Brauer relations. In our case, we have just shown that the $(4,4)$-isogeny is also Brauer verifiable. Moreover, we have shown:
\begin{prop}
    Let $f\in K[X]$ be a good quartic with critical elliptic curve $E_f$ and let $X_f\to\P^1_x$ be the Galois closure of $f(X)-x$. The isogeny $\jac(X_f)\sim E_f\times(E'_f)^3$ is Brauer verifiable.
\end{prop}

\subsection{The Prym curve}
We now work out a Weierstrass model for our Prym elliptic curve $E'_f$.
\begin{lem}\label{prymdata}
    Let $f\in K[X]$ be a good quartic with $f'(X)=4(X^3+AX+B)$ and let $j=j(Z_f)$. Then $E'_f$ has $j$-invariant \begin{align}\label{jprym}
        j(E'_f)=\frac{j^2}{4(j-1728)}
    \end{align}and the Weierstrass models $y^2=x^3-Ax^2-B^2$ and $y^2=\prod_{X_i\in Z_f}(x-X_i^2)$.
\end{lem}
\begin{proof}
    The off-diagonal self-fiber product is \begin{align}\label{sfp}
        \frac{f(X)-f(Y)}{X-Y} = X^3+X^2Y+XY^2+Y^3+2A(X+Y)+4B = 0. 
    \end{align}Let $v=X+Y, \ w=X-Y$; then \eqref{sfp} becomes $$v^3+vw^2+4Av+8B=0.$$ Letting $W=vw$, the normalization is \begin{align}\label{4model}
        W^2=-v(v^3+4Av+8B),
    \end{align}so setting $x=-\frac{2B}v,\ y=\frac{BW}{v^2}$ gives the Weierstrass model $$E'_f:y^2=x^3-Ax^2-B^2.$$
    Since $\prod_{X_i\in Z_f}X_i^2=B^2$, $-\sum_{X_i\in Z_f}X_i^2=-(\sum X_i)^2+2\sum_{i<j} X_iX_j=2A$ and $\sum_{i<j} X_i^2X_j^2=A^2$, the substitution $x\mapsto x+A$ sends it to the other model above.
    Using the standard formula \cite[\S 1.2]{tate} for the $j$-invariant we get \begin{align}j(E'_f)=\frac{2^{12}A^6}{\Delta(E'_f)}=-\frac{2^{12}A^6}{16B^2(4A^3+27B^2)}=-\frac{j^2(4A^3+27B^2)}{(432B)^2}=-\frac{j^2}{4(1728-j)}.\label{j'(A,B)}\end{align}
\end{proof}

\begin{reml}
    In \cite[\S 5.4]{kumar}, Kumar gives a plane model of the Humbert surface of discriminant $16$ in terms of the $j$-invariants of the two elliptic curves to the product of which the parametrized genus-$2$ Jacobian is $(4,4)$-isogenous. Our pair $(j(E_f), \ j(E'_f))$, as a function of $j=j(Z_f)$, describes a curve in Kumar's $(r,s)$-plane. Using \eqref{jprym} and \cite[Lemma 4.3]{nac} one can compute, after substituting $$t=j, \quad s=\frac{2016-t}{256}, \quad
    r=-\frac{9(t-2016)^2}{4096(t-1728)},$$ that this curve is $9s^2-16rs+18r=0$.
\end{reml}

Say that we wanted to make the $4$-congruence $\iota$ of Lemma \ref{4-iso} more explicit. It is easier to see that the two curves admit a $2$-congruence, as both are $2$-congruent to the critical point curve $\e_f:v^2=\prod_{X_i\in Z_f}(u-X_i)$. Indeed, we have $E_f[2]=\{\O,(-f(X_i),0)\}$ by Lemma \ref{crit} and $E'_f[2]=\{\O,(X_i^2,0)\}$ by Lemma \ref{prymdata}. The restriction of $\iota$ on $2$-torsions is generically compatible with these $2$-congruences:
\begin{lem}\label{compatibility}
    Let $f\in K[X]$ be a good quartic such that $\Gal(L/K)\simeq S_3$, where $L$ is the splitting field of $f'$. Then, the only $2$ congruence between $E'_f$ and $E_f$ sends $X_i^2$ to $-f(X_i)$
\end{lem}
\begin{proof}
    Let $\tilde\iota$ be the restriction of our $4$-congruence on $2$-torsions and let $L$ be the splitting field of $f'$ over $K$. As an isomorphism of Galois modules, $\tilde\iota$ maps identity to identity. On the $x$-coordinates of points of order $2$, say that $\tilde\iota(X_i^2)=-f(\tau(X_i))$ for some $\tau\in S_3$. By Galois compatibility, for any $\sigma\in G_K$ we have $$-f(\tau(^\sigma X_i^2))=\tilde\iota(^\sigma X_i)=-f(^\sigma\tau(X_i)),$$so $\tau\in Z(\mathrm{Gal}(L/K))=Z(S_3)=\{\mathrm{id}\}$.
\end{proof}
A different way to see that this $2$-congruence extends to a $4$-congruence is as follows: recall that, under our normalization, $f(X)=(\psi_{3,\e_f}(X)+A^2)/3=\frac14Xf'(X)+AX^2+3BX$. Therefore,\begin{align}
    f(X_j)-f(X_i)=A(X_j^2-X_i^2)+3B(X_j-X_i)=(X_j-X_i)(-AX_k+3B)=\notag\\=(X_j-X_i)(X_k^3-4X_iX_jX_k)=(X_j-X_i)X_k((X_i+X_j)^2-4X_iX_j)=(X_i^2-X_j^2)(X_i-X_j)^2.\label{alg.id}
\end{align}For a curve $C:y^2=\prod_{i=1}^3(x-e_i)$, the halves of $T_i=(e_i,0)$ are the points with $x=e_i+u_ju_k$ and $y=\pm u_ju_k(u_j+u_k)$, where $u_j$ and $u_k$ satisfy $u_j^2=e_i-e_j$ and $u_k^2=e_i-e_k$. Choose square roots $\delta_{ij}$ of $X_i^2-X_j^2$ and observe that $(X_i-X_j)\delta_{ij}$ is a square root of $-f(X_i)+f(X_j)$ by \eqref{alg.id}. Hence, a lift of $\iota|_{E'_f[2]}$ to the $4$-torsions must send $\delta_{ij}\mapsto\varepsilon\,(X_i-X_j)\delta_{ij}$ for a common $\varepsilon\in\{\pm1\}$ (again under the assumption that $f'$ has Galois group $S_3$), and we will find out which of the two possibilities $\iota$ is in Lemma \ref{preserve}. Moreover, the previous argument shows that for any $2$-congruence $X_i^2\mapsto-f(\tau(X_i))$, the condition $$\frac{f(\tau(X_j))-f(\tau(X_i)))}{X_i^2-X_j^2}\in (K^\times)^2$$is equivalent to lifting to a $4$-congruence; so, for $\tau\neq\mathrm{id}$, one would need $\frac{X_{\tau(i)}^2-X_{\tau(j)}^2}{X_i^2-X_j^2}$ to be a square for all $i\neq j$.
\begin{reml}\label{Q-comp}
    This strengthens Lemma \ref{compatibility}, for example extending it to all good quartics over a totally real number field such that $\e_f[2](K)=\{\O\}$. Indeed, under the notation of the lemma, in this case we have $\Gal(L/K)\simeq A_3$, so $\tau$ is a $3$-cycle, say $i\mapsto i+1\pmod3$ without loss of generality. Thus, we would have a solution $(X_1,X_2,X_3,U,V)\in L^5$ to $$X_{1}^2-X_3^2=U^2(X_3^2-X_2^2), \quad X_3^2-X_2^2=V^2(X_2^2-X_1^2)$$ with the first three entries pairwise distinct, implying $U,V\neq 0$ since $-1$ is not a square. Substituting the second equation in the first then shows that $X_1^2$ is strictly contained between the other two squares, contradicting both equations.
\end{reml}

Observe that, even if our family of derivatives $\{4(X^3+AX+B)\}$ is a two-parameter one, the (hyper)elliptic fibrations $$E_f:y^2=\mathrm{Disc}_X(f(X)-x), \quad \c_f:-Y^2=h_f(X), \quad E'_f:y^2=x^3-Ax^2-B^2,$$ when taken up to $K$-isomorphism, are one-parameter families over the $j$-line for $j=j(Z_f)$ (punctured at $j=0,1728,\infty$): for $\e_f$, one observes that, generically, good quartics with $j(Z_f)=j$ are linearly equivalent over $K$, since a $\overline\Q$-automorphism of the critical point curve $\e_f$ preserves the $2$-torsion as a Galois module. This fails for $j=0,1728$, where one may need to go over a cubic extension because of the extra automorphisms (see the paragraph before \cite[Lemma 4.1]{nac}). For $E_f$, linear changes of variable twist it by a fourth power over $K$ in virtue of the properties of the discriminant, thus preserving its $K$-isomorphism class. They also preserve the $K$-isomorphism class of $\c_f$: letting $g(X)=f(cX+d)=:f\circ\alpha(X)$, we have $$f(X)-x_i=(X-X_i)^2(X-X_{i,1})(X-X_{i,2})$$ with $x_i$ the critical values of $f$, and hence of $g$. Therefore, \begin{align*}g(X)-x_i= f\circ\alpha(X)-x_i=(\alpha(X)-X_i)^2(\alpha(X)-X_{i,1})(\alpha(X)-X_{i,2})=\\=c^{4}\left(X-c^{-1}(X_i-d)\right)^2\left(X-c^{-1}(X_{i,1}-d)\right)\left(X-c^{-1}(X_{i,2}-d)\right),
\end{align*}so, by \eqref{qsplit}, the curve $\c_g$ is isomorphic to $\c_f$ via $(X,Y)\mapsto(cX+d,c^3Y)$. Hence, linear changes of variable must preserve the $K$-isomorphism class of $E'_f$ as well, in light of \eqref{prym}. Linear rescalings twist the critical elliptic curve by the sixth power of the linear coefficient, thus preserving its $K$-isomorphism class, and clearly leave the precritical genus-$2$ curve unchanged. 

In virtue of the previous paragraph, let us from now on interchange freely, when $j\neq0$, between the notation with subscript $f$ to that with subscript $j$ for our families of curves and their maps. Set\begin{align*}
    u:=\frac{A^3}{B^2}=\frac{27j}{4(1728-j)}
\end{align*}and notice that $j\mapsto u$ is a bijection. We recall once again that the map $j\to(E_j,\psi_j)$ gives an isomorphism $\A^1\setminus\{1728\}\simeq Y_0(3)$. The Prym family is not modular, as can be seen by observing how its $j$-invariant \eqref{jprym} is not any of those in \cite[Table 3]{suzy} (even up to a linear change of variable). 

When $j\neq0$, letting $x\mapsto Ax$ in the model of Lemma \ref{prymdata} gives for $E'_j$ the model $y^2=ux^3-ux^2-1$, and with the additional change of variable $x\mapsto x/u$ we find the $K$-representative $E'_u:y^2=x^3-ux^2-u^2$ for the Prym family. Consider also the $(-1)$-twisted family $$F_j:=E_j^{' \ (-1)}:y^2=x^3+ux^2+u^2.$$
Observe that $F_j$ carries the rational section $Q_j=(0,u)$. As $j(F_j)$ has degree $2$ over the $j$-line, it does not factor through any of the maps $X_1(N)\to X(1)$ for $N>1,$ so the section is generically nontorsion. Observe moreover how $\jac(\c_u^{(-1)})$ has the $K$-rational section $P_j=[\infty^+ - \infty^-]$.
\begin{prop}\label{torsj}
    \begin{enumerate}[(i)]
        \item We have $Q_j=\pm P_j$ under the embedding $F_j\hookrightarrow\jac(\c_j)^{(-1)}$ obtained by twisting \eqref{prym}.
        \item If $j=0$, $f(X)$ is linearly equivalent to $g_B(X)=X^4+4BX$ for some $B\in K^*/(K^*)^3.$ The Prym curve is then $$E_{0,B}':y^2=x^3-B^2, \ \ B\in K^*/(K^*)^3$$ and the associated twist $F_{0,B}$ has the $3$-torsion point $P_{0,B}=(0,B)$.
        \item If $j\in\Q\setminus\{0,1728\}$, $Q_j$ is torsion if and only if $j\in\{-27648/11,55296/5\}$, in which case it is $5$- and $6$-torsion, respectively.
    \end{enumerate}
\end{prop}
\begin{proof}
    \begin{enumerate}[(i)]
        \item Recall that we have the map $\phi_j=\phi_f:\c_j\to E_j$ defined in Proposition \ref{cover}, sending both $\infty^+$ and $\infty^-$ to $0_{E_j}$ by Remark \ref{inf2}. Therefore, $P_j\in\jac(\c_j)(K(i))$ is sent to $0_{E_j}$ by the push forward $\phi_{j\,*}$; by \eqref{prym}, this means that $P_j$ lies in $E'_j\subset\jac(\c_j)$. We now find the coordinates $P_j$ on the the self-fiber product model \eqref{sfp} of $E_j$. Notice that we may assume $i\notin K$, since it is enough to show that $P_j=Q_j$ as sections over $\Q(j),$ which does not contain $i$.
        
        Let $p_3,p_4$ be lifts to $\c_j$ of the remaining coordinates in the fiber over $\P^1_x$ of $P_j$. Then, by \cite[proof of Lemma 3.1]{GLNZ}, $P_j$ corresponds under the embedding $E_j\hookrightarrow \c_j$ to the class $[p_3+i(p_4)-K_{\c_f}]$, where $K_{\c_j}$ is the canonical class. As the class of $P_j$ is $[\infty^+-\infty^-]=[2\infty^+-K_{\c_j}]$, we get that $p_3=i(p_4)$ is a point at infinity on $\c_j$. Therefore, $P_j$ is one of the three points at infinity given by normalization, which in homogenized coordinates for \eqref{sfp} are $[1:-1,0]$ and $[1:\pm i:0]$. As $\mathrm{Gal}(K(i)/K)$ acts nontrivially on it, $P_j$ is one of $[1:\pm i:0]$. Following the substitutions after \eqref{sfp}, we find (up to sign of $y$) $P_j=(0,iB)$ on the Weierstrass model of $E'_j$ from Lemma \ref{prymdata}, hence $P_j=(0,u)\in F_j$.
        \vspace{2mm}
    
        \item The first part of the statement immediately follows from the fact that linear variable changes $X\to cX$ send $g_B$ to $g_{Bc^{-3}}$ after renormalizing the leading coefficient. The second part follows directly from Lemma \ref{prymdata}, with the coordinates for $P_{0,B}\in F_{0,B}(K)$ following from the same argument as for (i). It is well-known that $P_{0,B}$ is $3$-torsion on $F_{0,B}$.
        \vspace{2mm}
    
        \item Suppose that $Q_j=(0,u)$ is torsion in $F_j$. Then $nQ_j=0_{F_j}$ for some $3\le n\le10$ or $n=12$, by Mazur's celebrated result \cite{mazurEisenstein} (and the fact that $Q_j$ is neither the origin nor $2$-torsion, since $F_j$ is in Weierstrass form). The formulas in \cite[\S 3.6]{silv} give, for the nine relevant $n$-division polynomials $\Psi_n=\Psi_{n,F_j}$ evaluated at $Q_j$, the following expressions:
        \begin{align*}
    \Psi_3(Q_j) &= 2^2u^3\\
    \Psi_4(Q_j) &= -2^5u^5\\
    \Psi_5(Q_j) &= -2^6u^8(u+4)\\
    \Psi_6(Q_j) &= -2^9u^{12}(u+8)\\
    \Psi_7(Q_j) &= -2^{12}u^{16}(u^2+4u-16)\\
    \Psi_8(Q_j) &= 2^{18}u^{21}(u^2+8u+8)\\
    \Psi_9(Q_j) &= 2^{20}u^{27}(u^3+12u^2+64u+192)\\
    \Psi_{10}(Q_j) &= 2^{25}u^{33}(u+4)(u^3+20u^2+80u-64)\\
    \Psi_{12}(Q_j) &= -2^{37}u^{48}(u+8)(3u^4+40u^3+160u^2+256u+512)
\end{align*}and hence, after clearing constants and powers of $u$ (which we can do since $j\neq0\iff A\neq0\iff u\neq0$),
    \begin{align*}
    \Psi_3^*(Q_j) &= 1\\
    \Psi_4^*(Q_j) &= 1\\
    \Psi_5^*(Q_j) &= u+4\\
    \Psi_6^*(Q_j)/\Psi_3^*(Q_j) &= u+8\\
    \Psi_7^*(Q_j) &= u^2+4u-16\\
    \Psi_8^*(Q_j)/\Psi_4^*(Q_j) &= u^2+8u+8\\
    \Psi_9^*(Q_j)/\Psi_3^*(Q_j) &= u^3+12u^2+64u+192\\
    \Psi_{10}^*(Q_j)/\Psi_5^*(Q_j) &= u^3+20u^2+80u-64\\
    \Psi_{12}^*(Q_j)/\Psi_6^*(Q_j) &= 3u^4+40u^3+160u^2+256u+512.
\end{align*}
        Since no polynomial of degree larger than $1$ appearing at the right-hand side has a rational root, $\Psi_n(Q_j)=0$ implies either $n=5\iff u=-4\iff j=-27648/11$ or $n=6\iff u=-8\iff j=55296/5$.
    \end{enumerate}
\end{proof}
\begin{reml}\label{badjs}
    The two twisted Prym curves over $\Q$ (excluding $j=0$) where the section becomes torsion, $F_{-27648/11}$ and $F_{55296/5}$, are the elliptic curves with Cremona labels 11a3 and 20a2, respectively. The former is $X_1(11)$. They have, among other interesting arithmetic properties, very small conductors. In \S\ref{5} we explore in more detail how conductors behave in our one-parameter families.\label{1stcurve}
\end{reml}

\section{Bad reduction and local arithmetic of the critical families}\label{5}
Let $\Delta=\Delta(f')=-(4A^3+27B^2)$ and let $\e_f:Y^2=f'(X)$ be the critical point curve. Our models have respective discriminants\begin{align}
    \Delta(\e_f)=16\Delta, \quad \Delta(E'_f)=16B^2\Delta, \quad \Delta(E_f)=16B^2\Delta^3\label{discs}
\end{align}by Lemma \ref{prymdata} and \cite[Proof of Lemma 4.3]{nac}.
\subsection{Selmer companions}

Using the LMFDB \cite{LMFDB} to examine some of the curves $E_f$ and $E'_f$---whose models can be constructed using the scripts in \cite[Critical elliptic curve, Prym elliptic curve]{github}---one observes that, for a prime $\ell>3$, the two curves appear to have the same reduction type (good, potentially good or potentially multiplicative). The following result confirms it:
\begin{lem}\label{badred}
    Let $f\in\Z[X]$ be a good quartic and let $\ell>3$ be a prime. $E_f$ and $E'_f$ have the same reduction type.
\end{lem}

\begin{proof}
By \eqref{discs}, primes of bad reduction for either curve are contained in $\{2\}\cup\{\ell:\ell\mid\Delta\}\cup\{\ell:\ell\mid B\}$. Recall that, given an integral Weierstrass equation for an elliptic curve, primes $\ell$ which divide the discriminant with exponent $v_\ell(\Delta)\notin12\Z$ are necessarily of bad reduction for the curve. One easily checks that bad places $\ell\neq2,3$ for $E_f$ are always bad for $E'_f$: clearly $\ell\mid\Delta(E_f)\Longrightarrow \ell\mid\Delta(E'_f)$, and if $\ell$ was good for $E'_f$ then necessarily $\ell\mid A,B$, since rescaling to another integral model makes the discriminant a $\ell$-adic unit. But then, $12\mid v_\ell(\Delta(E'_f))$ would imply $v_\ell(B)\ge3$ and $v_\ell(A)\ge2$, so rescaling $x=\ell^2X, \ y=\ell^3Y$ (and possibly iterating) gives a model for $E'_f$ minimal at $\ell$. From an equation for $E_f$ (say (20) in \cite{nac}) it is then immediate that $x=\ell^{4}X, \ y=\ell^{6}Y$ gives an integral model for $E_f$ minimal at $\ell$ for which the reduction is good. Viceversa, bad places $\ell\neq2,3$ for $E'_f$ are bad for $E_f$, since again we must have $\ell\mid A,B,$ so $v_\ell(B)\ge3$ and $v_\ell(A)\ge2$.

Recall that the places $\ell$ of potential multiplicative reduction for an elliptic curve with $j$-invariant $j$ are those such that $v_\ell(j)<0$. Equations \eqref{jcv} and \eqref{jprym} immediately give that the places where $E_f$ and $E'_f$ have potentially multiplicative reduction, with possibly the addition or exception of $2$ and $3$, are exactly $\{\ell:v_\ell(j-1728)>0\text{ or }v_\ell(j)<0\}$: this is because $\{\ell:v_\ell(j-1536)<0\}=\{\ell:v_\ell(j)<0\}$ and $\{\ell:v_\ell(j-1728)>0\}\cap\{\ell:v_\ell(j)>0\text{ or }v_\ell(j-1536)>0\}\subset\{2,3\}$. Since primes of bad reduction are either potentially good or potentially multiplicative, we are done.\end{proof}

Given an elliptic curve $E/K$, multiplication-by-$n$ gives a short exact sequence of $G_K$-modules
\begin{align}
\label{eq:kummer-short-exact}
  0 \longrightarrow E[n]
    \longrightarrow E
    \xrightarrow{[n]} E
    \longrightarrow 0.
\end{align}
Applying Galois cohomology to \eqref{eq:kummer-short-exact} yields the exact sequence
\begin{align*}
  0 \longrightarrow E(K)/nE(K)
    \xrightarrow{\kappa_{E,n}} H^1(K,E[n])
    \longrightarrow H^1(K,E)[n]
    \longrightarrow 0,
\end{align*}
where $\kappa_{E,n}$ is the Kummer map. For every place $v$ of $K$, the same construction over the completion $K_v$ gives
\begin{align*}
  0 \longrightarrow E(K_v)/nE(K_v)\longrightarrow H^1(K_v,E[n])\longrightarrow H^1(K_v,E)[n]\longrightarrow 0.
\end{align*}
Recall that the \(n\)-Selmer group of \(E/K\) is
\begin{align*}
    \mathrm{Sel}_n(E/K)=\ker\left(H^1(K,E[n])\longrightarrow\prod_v H^1(K_v,E)\right),
\end{align*}
where the \(v\)-component of the displayed map is the composite
\begin{align*}
    H^1(K,E[n])\xrightarrow{\mathrm{res}_v}H^1(K_v,E[n])\longrightarrow H^1(K_v,E).
\end{align*}
The group \(\operatorname{Sel}_n(E/K)\) is finite, see \cite{poonen} for more details.
In \cite{MR}, Mazur and Rubin introduce the notion of \textit{$n$-Selmer companions}, any two elliptic curves $E_1,E_2$  over a number field $K$ with the property that twisting $E_1$ and $E_2$ by any given quadratic character of $K$ yields two curves with isomorphic $n$-Selmer groups. They give some applications, some examples and a useful criterion for determining whether two curves are companions \cite[Theorem 3.1]{MR}. 

Until recently, only finitely many pairs of non-isogenous Selmer companions were known (where two pairs related by component-wise simultaneous quadratic twisting are considered the same). In \cite{Sp}, Spencer constructed an infinite family of such pairs by showing that, under certain congruence conditions on the coefficients for a Weierstrass model, a rational elliptic curve and its Hessian are $3$-Selmer companions. Here, we show the existence of infinitely many pairs of non-isogenous $2$-Selmer companions over $\Q$. We begin with a few useful lemmas (the proof of the first can be found in \cite{serre-local-fields}).
\begin{lem}\label{2can}
    Let $\ell>2$ and $u\in1+\ell\Z_\ell$. Then $u$ is a $2^k$-th power in $\Z_\ell^\times$ for all $k>0$. Moreover, if $u\in1+2^{k+2}\Z_2$ for some $k>0$, $u$ is a $2^k$-th power in $\Z_2^\times.$\label{ladic}
\end{lem}

\begin{lem}\label{preserve}
    Let $f(X)=X^4+2AX^2-8X$ with $|A|>7$, and let $\ell$ be a prime such that $E_f$ and $E'_f$ are potentially multiplicative over $\Q_\ell$. Then, $\iota$ identifies the canonical order-$4$ subgroups (cf$.$ \cite[Definition 2.3]{MR}) of $E_f$ and $E'_f$. Moreover, if $\e_f$ is also potentially multiplicative over $\Q_\ell$, its $2$-congruencewith any of the other two curves identifies the respective canonical order-$2$ subgroups.
\end{lem}
\begin{proof}
    We prove the second part first. \cite[Definition 2.3]{MR} and \cite[VII, Proposition 5.1]{silv} imply that the canonical $2$-subgroup is generated by the non-nodal Weierstrass point over the extension where the curve attains multiplicative reduction, corresponding to $z=-1\in\overline\Q_\ell^\times/q^\Z$ under Tate uniformization. By Remark \ref{Q-comp} the $2$-congruences commute, since $\e_f$ has no rational point of exact order $2$ by the rational root theorem, as $|A|>7$. As $B=-2$, \eqref{discs} implies that coalescence of two Weierstrass roots for either $\e_f, \ E_f$ or $E'_f$ is equivalent to coalescence of the corresponding two roots of the other curves under the respective $2$-congruences, so the first claim follows.

    For the canonical $4$-subgroup, let $X_3$ be the non-nodal Weierstrass root of $\e_f$ (and hence $X_3^2$ and $-f(X_3)$ for $E'_f$ and $E_f$ respectively). The canonical $4$-subgroup corresponds to $\mu_4\subset\overline\Q_\ell^\times/q^{\Z}$ under Tate uniformization, so it is generated by one of the two halves of the non-nodal Weierstrass point, i.e$.$ of $(X_3^2,0)$ (resp$.$ $(-f(X_3),0)$) for $E'_f$ (resp$.$ $E_f$). Recall that for a curve $C:y^2=\prod_{i=1}^3(x-e_i)$, the halves of $T_i=(e_i,0)$ are the points with $x=e_i+u_ju_k$, where $u_j$ and $u_k$ satisfy $u_j^2=e_i-e_j$ and $u_k^2=e_i-e_k$. We can distinguish the two halves at the Tate cusp $q=0$: there, Tate \cite[(46)-(48)]{tateellfun} gives uniformization $y^2+xy=x^3$ with smooth locus identified with $\mathbb{G}_m$ via $$z\mapsto\left(\frac{z}{(1-z)^2},\frac{z^2}{(1-z)^3}\right).$$ So, the canonical $4$-subgroup is generated by the image of $i=\sqrt{-1}$, and we have $x(-1)=-\frac14, \ x(i)=-\frac12$. For the non-nodal point we then have $u_ju_k=-\frac14,$ but also $u_j^2=u_k^2=-\frac14$ since the node has $x=0$, so $u_j=u_k$. Recall the notation from the paragraph after Lemma \ref{compatibility} and let $\iota$ act on $E'_f[4]$ by sending $\delta_{3j}$ to $\eta_{3j}=\varepsilon\,\delta_{3j}(X_3-X_j)$. Then, in residue characteristic $\ell>2$, identification of the $4$-canonical subgroups is equivalent to $$1=\eta_{31}/\eta_{32}=\left.\frac{\varepsilon\,\delta_{31}(X_3-X_1)}{\varepsilon\,\delta_{32}(X_3-X_2)}\right\vert_{q=0}=\left.\frac{(X_3-X_1)}{(X_3-X_2)}\right\vert_{q=0},$$ which follows from the coalescence of $X_1$ and $X_2$.

    If $\ell=2$, the sign choice cannot help in distinguishing the halves. Yet, we still have $u_j=u_k$ at the cusp, and this is enough for the distinction: for our generic curve $C$, let $r=\frac{u_j}{u_k}$ and $\lambda=\frac{e_2-e_1}{e_3-e_1}$, so that $C$ has Legendre model $y^2=x(x-1)(x-\lambda)$ with $j$-invariant \begin{align}
    2^8\frac{(1-\lambda+\lambda^2)^3}{\lambda^2(\lambda-1)^2},\label{jleg}
    \end{align}and we have $r^2=(1-\lambda)^{-1}$. Since the reduction is potentially multiplicative, we have $v(j)<0$, where $v$ is an extension of the valuation $v_2$ of $\Q_2$ to the splitting field $L$ of $C/\Q_2$. Moreover $v(\lambda)>0$ since $e_1$ and $e_2$ coalesce, so we get $v(\lambda)>4v(2)$ from \eqref{jleg}, and hence $v(r^2-1)=v(\lambda)>2v(2)$. As $(r+1)-(r-1)=2$, exactly one of them has valuation $>v(2)$, and since $u_j=u_k$ at the cusp, we must have $v(r-1)>v(2)$. In our case we have $$\frac{\eta_{31}}{\eta_{32}}=\frac{X_3-X_1}{X_3-X_2},$$ so identification of the canonical subgroups is equivalent to $$v(2)<v\left(\frac{X_3-X_2}{X_3-X_1}-1\right)=v(X_1-X_2)=v(\sqrt\Delta)=v(2)+\frac12v(A^3+27),$$ and we are done since $A$ is odd.
\end{proof}

\begin{lem}\label{5.6rep}
    Let $F$ be a finite extension of $\Q_\ell$ for some rational prime $\ell$, let $E_1,E_2$ be non-split potentially multiplicative elliptic curves over $F$ with the same quadratic splitting field $L/F$, and let $\iota:E_1[2^k]\simeq E_2[2^k]$ identify the canonical order-$2^k$ subgroups (cf$.$ \cite[Definition 2.3]{MR}). If the Tate parameters of both curves are non-norms from $L$, $\iota$ identifies the Kummer images $\kappa_{E_i,2^k}(E_i(F)/2^kE_i(F))$ inside the common $H^1(F,E_i[2^k])$.
\end{lem}
\begin{proof}
    This follows directly from \cite[Lemma 5.6]{MR}, along with Definitions 2.3 and 5.5 in the same paper.
\end{proof}

\begin{thm}\label{companions}
    Let $A\in(-3+32\Z)\setminus\{-3\}$. For the quartic $f(X)=X^4+2AX^2-8X$, the elliptic curves $E_f$ and $E'_f$ are $2$-Selmer companions over any number field.
\end{thm}
\begin{proof}
With our choice, we have \begin{align}\label{jb=2}
j(\e_f) = 1728\frac{A^3}{A^3+27},\quad
j(E_f)  = -\frac{A^3(A^3-216)^3}{(A^3+27)^3},\quad
j(E'_f) = -\frac{16A^6}{A^3+27}.
\end{align} 
We want to apply Theorem 3.1 in \cite{MR}, setting $p=2$ and $k=1$. We know from Lemma \ref{4-iso} that there is a $G_K$-invariant isomorphism $\iota:E_f[4]\simeq E'_f[4]$, so condition (i) in the theorem is satisfied. Lemma \ref{preserve} shows that (iii) holds, so for (ii) and (iv) to be satisfied it is enough to show that $E_f$ and $E'_f$ have the same set $S$ of primes of potential multiplicative reduction, and $2\in S$.
Observe that $v_3(j(E'_f))<0\iff v_3(j)<0$ by \eqref{jb=2} and $v_3(j(E_f))<0\iff v_3(j)<0$ or $v_3(j-1728)>6$ by \eqref{jcv}. Since for us $v_3(j-1728)=6-v_3(A^3+27)$ by \eqref{jb=2}, the two conditions are equivalent, so $3$ is a prime of potentially multiplicative reduction for $E_f$ if and only if it is for $E'_f$. Then, thanks to Lemma \ref{badred}, the requirement reduces to $2$ being of potentially multiplicative reduction for both curves. By \eqref{jb=2}, potential multiplicativity at $2$ is equivalent to $v_2(A+3)>4$, which follows from our hypothesis on $A$.
\end{proof}

The infinitude result from the Introduction now follows easily:

\begin{proof}[Proof of Theorem \ref{infcomp}]
    Let us restrict to a maximal subset of the $A$ found in Theorem \ref{companions} for which the function $A\mapsto (j(E_j),j(E'_j))$ is injective. Clearly such a set is infinite by \eqref{jb=2}. As $E_j$ and $E'_j$ are non-isogenous over $\Q(j)$ ($j(E'_j)$ has degree $2$ in $j$, while an Hauptmodul for some $X_0(N)$ must have degree at least $3$ over the $j$-line), $E_f\sim E'_f$ happens over $\Q$ for finitely many values of $j$ (there are only finitely many possible degrees for cyclic isogenies over $\Q$, by the work of Mazur \cite{mazurIsogenies}), and each such value precludes finitely many $A$, so the claim follows.
\end{proof}

\begin{lem}\label{semistable}
    Let $f$ be as in the statement of Theorem \ref{companions} with $A$ coprime to $3$. The curves $E_f$ and $E'_f$ are semistable of the same conductor.\end{lem}
    
    \begin{proof}
    This is equivalent to the two curves having the same bad primes, all of multiplicative reduction.
    The only primes $\ell$ dividing their respective discriminants are $2$ and those dividing $\Delta$, and they are all of multiplicative reduction: this is tested by checking that $\ell$ divides the discriminant of a minimal model over $\Q_l$ and the quantity $c_4$ is an $\ell$-adic unit \cite[VII, Proposition 5.1]{silv}. One computes, by using Tate's \textit{Formulaire} \cite{tate} 
    (or by recalling $j=\frac{c_4^3}{\Delta}$ and using \eqref{jcv}, \eqref{jprym} and \eqref{discs}):\[c_4(E_f)=16A(A^3-216),\quad c_4(E'_f)=16A^2.\]So the models are automatically minimal for $\ell>2$ and, for both curves, if $\ell$ divides both the discriminant and $c_4$ then clearly $\ell\in\{2,3\}$, so $\ell=2$ by the hypothesis on $A$. Notice that $2^4\mid c_4$, so the given models may not be minimal over $\Q_2$; still, $2$ is of multiplicative reduction for $E'_f$, as can be seen by computing $c_4=A^2$ for the model $$E'_f: Y^2+XY+Y=X^3+\frac{5-A}{4}X^2+\frac{1-A}{4}X-\frac{A+3}{16}$$obtained through the change of variables $x=4X+2, \ y=8Y+4X+4$, which is integral over $\Q_2$ since $16\mid A+3$ (and minimal, since $16\nmid c_4$). The same holds for $E_f$, since $c_4=A(A^3-216)$ for the model $$E_f: Y^2+XY+Y=X^3+\frac{1-A^2}{2}X^2+\frac{A^4-4A^2+72A-5}{16}X+\frac{A^4-8A^3-2A^2+72A-447}{64}$$ obtained via $x=4X+1, \ y=8Y+4X+4$---again integral and minimal over $\Q_2$.
\end{proof}
\begin{reml}
    More generally, $c_4(E_f)$ is an $\ell$-adic unit for all bad primes $\ell>3$ under the weaker hypotheses that $(A,B)=1$.\label{c4unit}
\end{reml}

\begin{reml}\label{cpcurve}
    For $f$ as in Lemma \ref{semistable} and with $v_2(A+3)>6$, numerical experiments hint to $\e_f$ having conductor $N$ as well. We can show this as follows: $\e_f$ has discriminant $16\Delta$, so bad primes $\ell>2$ divide $A^3+27$, and are all of multiplicative reduction since $c_4(\e_f)=-48A$ and $(A,3)=1$. Therefore, we just need to show that $v_2(N(\e_f))=1$. We also compute $c_6(\e_f)=2^63^3$ with Tate's \textit{Formulaire} \cite{tate}, and hence $2^{18}\mid c_4^3-c_6^2=-2^{12}3^3(A^3+27)$. By \cite[2.1]{stein} (for $p=2$) and our hypothesis on $v_2(A+3)$, $\e_f$ is not minimal for $l=2$, and a minimal model has \[
    c_4(\e_f^{\min})=\frac{c_4}{2^4},\quad
    c_6(\e_f^{\min})=\frac{c_6}{2^6},\quad
    \Delta(\e_f^{\min})=\frac{\Delta(\e_f)}{2^{12}}.
    \] Using \cite[2.1]{stein}, we find that $v_2(N(\e_f))=1\iff 2\mid\Delta(\e_f^{\min}),$ which is precisely ensured by $v_2(A+3)>6$.
\end{reml}

In light of Remark \ref{cpcurve}, it is natural to ask whether the $2$-congruences between $\e_f$ and each of $E_f,E'_f$ also induce isomorphisms of $2$-Selmer groups. This turns out to indeed be the case, but only after a quadratic twist. Recall that $F_f$ denotes the $(-1)$-twist of $E'_f$.
\begin{prop}\label{2self'}
Let $f(X)=X^4+2AX^2-8X$, where $A=2^hd-3$ with $h\ge7$ odd and $3\nmid d\equiv1\pmod4$. We have  $$\mathrm{Sel}_2(E^{(-1)}_f/\Q)\simeq\mathrm{Sel}_2(\e^{(-1)}_f/\Q)\simeq\mathrm{Sel}_2(F_f/\Q).$$
\end{prop}
\begin{proof}
    Let us work with $\e_f^{(-1)}$ and $F_f$: the proof will then extend to $E_f^{(-1)}$ by the equality of local $2$-Kummer images provided by Theorem \ref{companions}, which has weaker hypotheses. Looking at the hypotheses of \cite[Theorem 6.1]{MR}, we see that we are lacking conditions (i) and (iii), with (i) being replaced by the $2$-congruence $-X_i\mapsto-X_i^2$. Lemma \ref{preserve} shows that the $2$-congruence induces an identification of the canonical $2$-subgroups.

    Now go through the proof of Theorem 6.1: their \textit{Case 2} and \textit{Case 5} are empty for us, and \textit{Case 3} can be followed verbatim since it only requires the $2$-congruence. Therefore, we are left with proving that the $\ell$-adic $2$-Kummer images coincide for $\ell=\infty$ (their \textit{Case 1}) and $\ell$ of potentially multiplicative reduction (their \textit{Case 4}).
    
    \textit{Bad primes:} we first prove directly that our curves have the same splitting field $\Q_\ell(\sqrt{-c_6})$ over $\Q_\ell$, replacing Mazur and Rubin's Lemma 5.7(i): equality of splitting fields is equivalent to$$\gamma_f:=\frac{c_6(F_f)}{c_6(\e^{(-1)}_f)}\in(\Q_\ell^\times)^2.$$We have $c_6(F_f)=-2^6(A^3+54)$ and $c_6(\e_f^{(-1)})=-2^63^3$, so $\gamma_f=\frac1{27}(A^3+54)$, which is a square by Hensel lifting since $A^3+27=0\pmod\ell$ (or, for $\ell=2$, because $(A^3+54)/27\equiv1\pmod8$). This automatically settles the case of split multiplicative reduction, where the proof simply uses identification of $2$-canonical subgroups on top of equality of splitting field to apply \cite[Lemma 5.4]{MR}.

    For the nonsplit case, following Mazur and Rubin's notation we first have to show that either both Tate parameters are norms from the splitting extension, or none is.
    Let $q$ and $q'$ be the Tate parameters of $\e^{(-1)}_f$ and $F_f$; the usual modular identity for the $j$-invariant gives \[j=\frac1q+744+196884q+O(q^2), \quad j(F_f)=\frac{1}{q'}+744+196884q'+O((q')^2),\]so $q=j^{-1}(1+\frac{744}{j}+\frac{750420}{j^2}+...)$ and the same for $q'$. The factors in parentheses are in $1+\ell\Z_\ell$ since $v_\ell(j), \ v_\ell(j(F_f))<0$, and for $\ell=2$ they are in $1+8\Z_2$ in virtue of our hypothesis on $v_2(A+3)$ and the fact that $4\mid744$ and $2\mid750420$. Therefore, by Lemma \ref{ladic}, they are squares in $\Q_\ell^\times$. Using \eqref{jb=2} we get $$\frac{q}{q'}\in\frac{j(F_f)}{j}(\Q_\ell^\times)^2=-\frac{A^3}{108}(\Q_\ell^\times)^2=(\Q_\ell^\times)^2,$$ either by $A^3+27=0$ for $\ell>2$ or $A\equiv5\pmod8$ for $\ell=2$. 
    Now, set $q/q'=u^2, u\in \Q_\ell$ and let $L=\Q_\ell(\sqrt{3})$ be the common splitting
    field of $\e^{(-1)}_f$ and $F_f$. Then, $q$ is a norm from $L^\times$ if and only if $q'$ is, since $u^2=N_{L/\Q_\ell}(u)$. Therefore, if they are not norms, we are done by Lemma \ref{5.6rep} for $k=1$. Moreover, we have $j^{-1}\in3A(A^3+27)(\Q^\times)^2\subset-6d(\Q_2^\times)^2$ by \eqref{jb=2}, so the hypothesis $d\equiv1\pmod4$ is equivalent to $j$---and hence $q$---not being a norm from $\Q_2(\sqrt3)$.
    
    Now suppose that $q, \ q'$ are norms, so in particular $\ell>2$. Then, the splitting extension $L/\Q_\ell$ is unramified and hence $v_\ell(q)=-v_\ell(j)\equiv0\pmod2$, so $v_\ell(A^3+27)$ is even. Choose
    $\beta\in L^\times$ with $N_{L/\Q_\ell}(\beta)=q$ and put $\beta'=\beta/u$, so that $N_{L/\Q_\ell}(\beta')=q'$. The additional generator of the $2$-Kummer image provided by \cite[Lemma 5.6(ii)]{MR} is built by choosing square roots $\alpha^2=\beta$ and $\alpha'^2=\beta'$ in $\bar\Q_\ell$. So, for the local $2$-Kummer images to coincide, it is enough that $u\in (L^{\times})^2$: in that case $\alpha/\alpha'\in L$, so the additional generator for $\e_f^{(-1)}$ would differ from that for $F_f$ by the generator of the common subgroup of \cite[Lemma 5.6(i)]{MR} obtained by uniformizing $\sigma\to\left(\frac{(\alpha\alpha')^{\sigma\psi(\sigma)}}{\alpha\alpha'}\right)$ for either curve. So, since $L$ is unramified, we just need to check that the valuation of $u$ is even. Notice that $v_\ell(q)=-v_\ell(j)=v_\ell(A^3+27)$ and, similarly, $v_\ell(q')=v_\ell(A^3+27)$; since $v_\ell(u)=\frac12v_\ell(q/q')$, we are done. 

    \textit{$\ell=\infty$:} we have to show equality of the $2$-Kummer images $$C^{(-1)}(\R)/2C^{(-1)}(\R)\to H^1(G_{\R},C^{(-1)})$$for $C=\e_f, \ E'_f$, under the identification induced by their $2$-congruence. The number of real components of $\e^{(-1)}_f$ is one if $f'$ has one real root and two otherwise.
    The number of real components of $E^{(-1)}_f$ and $F_f$ is the same as that of $\e^{(-1)}_f$ in virtue of the $2$-congruence. If there is only one component, doubling is surjective on it, so we are done as the $2$-Kummer image is trivial. Now assume that $f'$ has three real roots $X_1<X_2<X_3$, so we have \begin{align}
    X_1+X_2+X_3=0,\quad  X_1X_2X_3=2.\label{roots}\end{align}
    Standard $2$-descent gives the $(x-T)$ map \cite[\S 5]{PS}$$\phi_f:\e_f^{(-1)}(\R)/2\e_f^{(-1)}(\R)\to (\R^\times/(\R^\times)^2)^3, \quad [(X,Y)]\mapsto ([X+X_1],[X+X_2],[X+X_3]),$$ so the real $2$-Kummer image is generated by a point in the bounded component $-X_3\le X\le -X_2$, having signature $\mathrm{sgn}(\phi_f([(X,Y)]))=(-,-,+)$. As the twisted $2$-congruence identifies $-X_i$ with $-X_i^2$, we need to show that we get the same signature for the real $2$-Kummer image of $$F_f(\R)/2F_f(\R)\to (\R^\times/(\R^\times)^2)^3,$$that is, that $X_3^2>X_1^2,X_2^2$. The two identities in \eqref{roots} imply $X_2<0<X_3$ and $X_3=|X_1+X_2|$, so we are done.
\end{proof}

\begin{reml}
    The reason why twisting by $-1$ is useful is in proving equality of real $2$-Kummer images, where the argument breaks down for the correspondence $X_i\to X_i^2$. Moreover, the twist with the marked nontorsion point is precisely $F_f$, making the statement of Proposition \ref{2self'} useful for exhibiting nontrivial Tate-Shafarevich groups. Still, the same techniques prove the $2$-Selmer isomorphism even for the untwisted curves (and the same proof at $\ell=2$ works even without the congruence condition hypothesis on $d=2^{-v_2(A+3)}(A+3)$, since the splitting extension is unramified), under the additional hypothesis that $A>-3$, which ensures that there is a single real component. We remark that numerical evidence makes it likely that the hypotheses of odd exponent at $2$ and coprimality with $3$ can also be removed, but on the other hand the stronger claim of full $2$-Selmer companionship with $\e_f$ can be easily disproved experimentally.
\end{reml}

\subsection{Tate-Shafarevich groups and visibility}
Recall that, for any elliptic curve $E$ over a number field $K$ and any positive integer $n$, we have the exact sequence \begin{align}
    0\to E(K)/nE(K)\to\mathrm{Sel}_n(E/K)\to\Sha(E/K)[n]\to0,\label{selsha}
\end{align}where the Tate-Shafarevich group $\Sha(E/K)$ is defined as\[\ker\left(H^1(K,E)\xrightarrow{\mathrm{res}_v}\prod_v H^1(K_v,E)\right).\]Therefore, if two elliptic curves with isomorphic $2$-Selmer groups have trivial $2$-torsion and different ranks, one of them is forced to have nontrivial Tate-Shafarevich group. Set $r=\mathrm{rk}_\Q(F_f)$ and take all ranks over $\Q$.
\begin{cor}
    Let $f(X)=X^4+2AX^2-8X$ with $A\in\Z$ as in the statement of Theorem \ref{companions}, and let $d\in\Q^\times/(\Q^\times)^2$. We have \begin{align}\label{rk+sha}
        \mathrm{rk}(E^d_f)+\dim_{\F_2}\Sha(E^d_f/\Q)[2]=\mathrm{rk}((E'_f)^d)+\dim_{\F_2}\Sha((E'_f)^d/\Q)[2].
    \end{align}Moreover, if $A$ is as in the statement of Proposition \ref{2self'}, the integer
    \begin{align*}
        \mathrm{rk}(C^{(-1)})+\dim_{\F_2}\Sha(C^{(-1)})[2]
    \end{align*}is the same for $C=\e_f, E_f, E'_f$.
    In particular, for $C=E_f$ and $A$ as in the statement of Theorem \ref{companions}, and for $C=\e_f$ and $A$ as in the statement of Proposition \ref{2self'}, if $\mathrm{rk}(C^{(-1)})=0$ we have\begin{align}\label{sharank}
        \frac{|\Sha(C^{(-1)}/\Q)[2]|}{|\Sha(F_f/\Q)[2]|}\ge2.
    \end{align}\label{sha}
\end{cor}

\begin{proof}
    By Theorem \ref{companions} we have $\mathrm{Sel}_2(E^d_f/\Q)\simeq\mathrm{Sel}_2((E'_f)^d/\Q)$ for any $d\in\Q^\times/(\Q^\times)^2$, so \eqref{rk+sha} follows, provided both curves have rational $2$-torsions of the same size; this is indeed the case: we have $E_f^{(-1)}(\Q)[2]=F_f(\Q)[2]=\{\O\}$ as both $2$-torsions are $G_{\Q}$-isomorphic to $\e_f[2]$, which we know has no finite rational point. The second assertion follows in the same way from Proposition \ref{2self'}. In virtue of either of our choices of $A$, $j(Z_f)=j(\e_f)$ is nonzero and has $2$-adic valuation $\le1$, so it does not belong to $\{0,-27648/11,55296/5\}$. Therefore, taking $d=-1$, the positivity of $r=\mathrm{rk}(F_f)$ ensured by Proposition \ref{torsj}(iii), together with our hypothesis on $\mathrm{rk}(C^{(-1)})$, implies\[\dim_{\F_2}\Sha(C^{(-1)}/\Q)[2]=r+\dim_{\F_2}\Sha(F_f/\Q)[2]\ge1.\]
\end{proof}

\begin{reml}\label{shaexp}
    Some of the first small values of $A$ as in Theorem \ref{companions} for which $\mathrm{rk}(E^{(-1)}_f)=0$ are $$A=-99,61,125,157,253.$$ For all of them we can check that $\mathrm{rk}_\Q(F_f)=2$ and $\Sha(E_f^{(-1)}/\Q)[2]\simeq(\Z/2\Z)^2$: for $A=-99$, $E_f^{(-1)}$ has Cremona label 485136h2, while $F_f$ is 485136b1; for $A=61$, $E_f^{(-1)}$ is 56752e2 and $F_f$ is 56752i1; for $A=125$, $E_f^{(-1)}$ has Cremona label 244144b2, while $F_f$ is 244144a1; for $A=157$, $E_f^{(-1)}$ is 101840j2 and $F_f$ is 101840k1, while for $A=253,$ $E_f^{(-1)}$ is 144592p2 and $F_f$ is 144592a1.
\end{reml}
\begin{reml}\label{shadiv}
    One can check on the respective LMFDB pages that the reported size of $\Sha_{\text{an}}(E_f^{(-1)}/\Q)$ is $16$ for all the values of $A$ considered in Remark \ref{shaexp}---except for $A=253$, for which it is $64$. Moreover, one can verify this for $\Sha$ rather than $\Sha_{\text{an}}$ by using the PARI code in \cite[selmer-sha]{github} (when generators have order up to $4$), which also works for cases not present in the LMFDB, like $A=-163$.
\end{reml}
In light of Corollary \ref{sha}, Remark \ref{shadiv} suggests that, any time the $2$-Selmer companionship forces $\Sha(E_f^{(-1)}/\Q)[2]$ to be nontrivial, it also forces it to be $2$-divisible. The natural explanation compatible with the data is the following refined version of \eqref{rk+sha} for $d=-1$:
\begin{align}
    2\mathrm{rk}_{\Q}(E^{(-1)}_f)+\log_2|\Sha(E^{(-1)}_f/\Q)[4]|=2\mathrm{rk}_{\Q}(F_f)+\log_2|\Sha(F_f/\Q)[4]|.\label{rk+sha4}
\end{align}In other words, we are led to conjecture that, under the hypotheses of Theorem \ref{companions}, our $4$-congruence induces an isomorphism of $4$-Selmer groups between the $(-1)$-twists. Moreover, a search using the script in \cite[Selmer isomorphism comparison]{github} suggests that the weakest condition which picks up this isomorphism is $A\equiv-3\pmod{32}$ or $A\equiv45\pmod{64}$, which we have verified for $|A|$ up to $60000$:

\begin{con}\label{4selcon}
    If $-3\neq A\equiv-3\pmod{32}$ or $A\equiv45\pmod{64}$, $\iota_{-1}$ induces an isomorphism $$\mathrm{Sel}_4(E^{(-1)}_f/\Q)\simeq\mathrm{Sel}_4(F_f/\Q).$$
\end{con}

\begin{reml}
The stronger claim of a $4$-Selmer companionship is falsified, even assuming natural hypotheses stronger than those of Proposition \ref{2self'}, by taking for example $A=125, \ d=-443$ \cite[Selmer isomorphism comparison]{github}. Many of the possibilities for the ranks and $4$-torsions of Sha allowed by \eqref{rk+sha4} turn out to happen: for instance, for $A=3037\equiv-3\pmod{32}$ we have $\mathrm{rk}_{\Q}(E^{(-1)}_f)=0, \ \mathrm{rk}_{\Q}(F_f)=4$ and $\Sha(E_f^{(-1)}/\Q)[4]\simeq(\Z/4\Z)^4$ \cite[selmer-sha]{github}.\end{reml}

\begin{reml}
    For $A=253=2^8-3$ it is reported that $F_f$ has rank $2$ and $\Sha_{\text{an}}(E^{(-1)}_f/\Q)$ has size $8^2,$ suggesting the existence of elements of order $8$ in the Tate-Shafarevich group of $E_f^{(-1)}$, and of an isomorphism $\mathrm{Sel}_8(E^{(-1)}_f/\Q)\simeq\mathrm{Sel}_8(F_f/\Q)$. While the latter does not have to be induced by an $8$-congruence---let alone one lifting our $4$-congruence---it is natural to ask whether there is a trend of $\Sha(E_f^{(-1)}/\Q)$ having elements of exact order an increasingly larger power of $2$ as $v_2(A+3)$ grows. We currently lack sufficient numerical evidence to express a conjecture in this direction. We also remark that the existence of elements of exact order an arbitrarily large power of $2$ in the Tate-Shafarevich group of some rational elliptic curve has only been known since the recent breakthroughs of Smith \cite{smith17} on $2^{\infty}$-Selmer ranks in quadratic twist families.
\end{reml}

We are able to prove Conjecture \ref{4selcon} under slightly stronger assumptions on $A$, which still yield an infinite family:
\begin{thm}\label{4sel}
    Let $A\in(-3+32\Z)\setminus3\Z$ be such that $v_\ell(A^3+27)$ is odd for all $\ell\mid A^3+27$ and $(A+3)2^{-v_2(A+3)}\equiv1\pmod4$. Then, for $f(X)=X^4+2AX^2-8X,$ we have $$\mathrm{Sel}_4(E^{(-1)}_f/\Q)\simeq\mathrm{Sel}_4(F_f/\Q).$$
\end{thm}
\begin{proof}
    We follow the proof structure of Proposition \ref{2self'}, still with reference to \cite[Theorem 6.1]{MR}, noting that we are again lacking hypotheses (i) and (iii) in their statement, with (i) being replaced by our twisted $4$-congruence (Lemma \ref{4-iso}). The latter maps the canonical $4$-subgroups \cite[Definition 2.3]{MR} to one another for every bad prime in virtue of Lemma \ref{preserve}. As in the proof of Proposition \ref{2self'}, \textit{Case 2} and \textit{Case 5} are empty for us (the former because we required $(A,3)=1$), and \textit{Case 3} is covered by the hypothesis we already have. So, we can focus on bad primes (\textit{Case 4}) and the place at infinity (\textit{Case 1}).

    \textit{Bad primes:} we already know that both curves have potentially multiplicative reduction, multiplicative at $\ell>2$, and that the splitting extensions $L=\Q_\ell(\sqrt3)$ agree for both curves. So, equality of $4$-Kummer images at split primes follows from \cite[Lemma 5.4]{MR}. Assume the reduction is not split. Recall the equation
    \[q=j^{-1}\left(1+\frac{744}{j}+\frac{750420}{j^2}+...\right)\]for the Tate parameter of an elliptic curve with $j$-invariant $j$. The Tate parameter of $E_f^{(-1)}$ (resp. $F_f$) modulo $(\Q_\ell^\times)^2$ is just $j(E_f)^{-1}$ (resp. $j(E'_f)^{-1}$) by Lemma \ref{ladic}. This also holds for $\ell=2$, because the $2$-adic valuation of both $j$-invariants is at most $-1$ by \eqref{jb=2}, in virtue of our hypotheses. By the same equation, their ratio is \[\frac{(A^3-216)^3}{16A^3(A^3+27)^2}\in\frac{(A^3-216)^3}{A^3}\Q_\ell^2=3^{12}\Q_\ell^2,\] hence a square; the odd exponent hypothesis implies that both Tate parameters are not norms for $\ell>2$, since in this case $L/\Q_\ell$ is unramified. Notice that the additional congruence condition on the odd part of $A+3$ implies that the Tate parameters are not norms at $\ell=2$ as well, as in the proof of Proposition \ref{2self'}. So, we are done by Lemma \ref{5.6rep} for $k=2$, since our $4$-congruence preserves the canonical $4$-subgroups.

    \textit{$\ell=\infty$:} If $\e_f$ has just one real component, then so do $E_f^{(-1)}$ and $F_f$ by the $2$-congruences. As multiplication-by-four is surjective on the unbounded component, both $4$-Kummer images are trivial, and we are done. Assume therefore that there are two real components and work with a generic elliptic curve $E$. \cite[Lemma 5.1]{MR} tells us that the Kummer image is generated by the classes $c_t\in H^1(G_\R,E[4])$ with $c_t:\sigma\mapsto t^{\sigma}-t$, $4t\in E[2]$. These cocycle classes are trivial for the point at infinity and for its half in the unbounded real component by the aforementioned surjectivity, so we may assume that $4t$ lies on the bounded component. By \cite[V, Theorem 2.3]{silv-advanced} we can write $E(\C)=\C^\times/q^{\Z}$ with $q\in(0,1)$ and complex conjugation acting as $z\mapsto\overline{z}$; let $\tau$ be the quotient map. A point $\tau(z)\in E(\C)$ is real when $\overline{z}=zq^k$ for some integer $k$, which forces $k=0$ and hence $z\in\R^\times$ by taking absolute values. The component containing the identity is the image of $\R^+$, so the bounded component is the image of $\R^-$. Let $t=\tau(z)$ with $4t$ on the bounded component, so that $z^4<0$. Setting $z=\rho\zeta$ with $\rho>0$ and $|\zeta|=1$ gives $\zeta^4=-1$, so $\overline{t}-t=\tau(\overline{z}/z)=\tau(\zeta^{-2})\in\{\tau(\pm i)\}$.

    Notice that this gives exactly the possible generators of the archimedean canonical $4$-subgroup, which we already know is preserved by $\iota$ at finite primes by Lemma \ref{preserve}. We therefore follow the proof of the lemma and its notation: Tate's formulas \cite[(46)-(48)]{tateellfun} have real coefficients when $q$ is real, and an isomorphism over $\R$ between $E$ and its Tate model rescales the differences $e_i-e_j$ by a positive factor, so we may determine their signs on the Tate model. The Weierstrass roots are then $x(\tau(-1))$ and $x(\tau(\pm\sqrt q))$, which stay distinct for $0<q<1$ and tend to $-\frac14$ and $0$ as $q\to0$, respectively. Thus $T_3:=\tau(-1)$ is the $2$-torsion point with the smallest $x$-coordinate $e_3$ for every such $q$, giving $u_j^2,u_k^2<0$. For the halves $\tau(\pm i)$ of $T_3$, the product $u_ju_k=x(\tau(i))-x(\tau(-1))$ is real, nonzero and tends to $-\frac14<0$ as $q\to0$, so $u_ju_k<0$. As a result, the condition $u_j=u_k$ of Lemma \ref{preserve} gets replaced by $u_j/u_k>0$.

    Now let $X_1<X_2<0<X_3$ be the roots of $f'$. The smallest Weierstrass roots of $F_f$ and $E_g^{(-1)}$ are $-X_3^2$ and $f(X_3)$ respectively: indeed, $f(X_j)-f(X_3)=(X_3^2-X_j^2)(X_3-X_j)^2>0$ by \eqref{alg.id}. As in the proof of Lemma \ref{preserve}, $\iota$ sends $\delta_{3j}$ to $\eta_{3j}=\varepsilon\,\delta_{3j}(X_3-X_j)$, where $\delta_{3j}^2=X_j^2-X_3^2$ and $\eta_{3j}^2=f(X_3)-f(X_j)$. It follows that
    \[\frac{\eta_{31}}{\eta_{32}}=\frac{\delta_{31}}{\delta_{32}}\cdot\frac{X_3-X_1}{X_3-X_2},\]
    and the last factor is positive since $X_3$ is the largest root. So $\delta_{31}/\delta_{32}>0$ implies $\eta_{31}/\eta_{32}>0$, which means that $\iota$ sends $\tau(\pm i)$ to $\tau(\pm i)$, and the two real $4$-Kummer images coincide.
\end{proof}

\begin{reml}\label{posden2}
    The set of integers $A$ for which the statement of Theorem \ref{4sel} holds has positive density in $\Z$. This follows by applying the following result of Hooley \cite{hoo}
    \begin{center}
        Let $g\in\Z[x]$ be a cubic polynomial with the property that for each prime $p$ there exists $n_p\in\Z$ such that $p^2\nmid g(n_p)$. Then, $g(n)$ is squarefree for a positive proportion of integers $n$.
    \end{center}to $g(x)=2^{-v_2(A+3)}(A^3+27)$ where, say, $A=32(3(4x+1)-2)-3$: as $g(0)=7\cdot109$, for all $p$ we can take $n_p=0$.
\end{reml}

In \cite{CM}, Cremona and Mazur study the concept of \textit{visibility} of elements in $H^1(G_K,E)\simeq\mathrm{WC}(E/K)$, showing how it is often the case that nontrivial elements of small order in the Tate-Shafarevich group of rational elliptic curves are \textit{visible} as curves in some abelian surface---usually the Jacobian of a genus $2$ curve obtained by gluing $E$ to another elliptic curve $F$ along the graph of a $n$-congruence $E[n]\simeq F[n]$. These ideas have been developed to show specific (in)visibility results for elements of small order in Sha, for instance in \cite{agashestein}, \cite{brfi} and \cite{fi}. Here, we show that the isomorphisms of Selmer groups of Theorems \ref{companions} and \ref{4sel} lead to visibility of any nontrivial Sha elements \textit{explained} by $F_f(\Q)$. In the remainder of this section, let $C$ be either $E_f$ or $\e_f$, and let $J$ be $\jac(\c_f^{(-1)})$ if $C=E_f$ or the abelian surface obtained by glueing $F_f$ with $\e_f^{(-1)}$ along $-X_i\mapsto-X_i^2$, and then quotienting by the common $2$-torsion (see \cite[Visibility and congruence moduli]{CM}). Recall the notation $r:=\mathrm{rk}(F_f)$ and set $s=\mathrm{rk}(C^{(-1)})$.

\begin{prop}\label{visibility}
    At least $\max(0, \ r-s)$ independent generators of $\Sha(C^{(-1)}/\Q)[2^k]/\Sha(C^{(-1)}/\Q)[2^{k-1}]$ are visible in $J$ in all of the following cases:
    \begin{enumerate}[(i)]
        \item $A$ as in Theorem \ref{companions}, $C=E_f$, $k=1$;
        \item $A$ as in Proposition \ref{2self'}, $C=\e_f$, $k=1$;
        \item $A$ as in Theorem \ref{4sel}, $C=E_f$, $k=2$.
    \end{enumerate}
\end{prop}
\begin{proof}
    Apply point b) in \cite[Visibility and congruence moduli]{CM} to $E=C^{(-1)}, \ B=F_f$ and $\iota_{-1}:C^{(-1)}[2^k]\simeq F_f[2^k]$ the $G_{\Q}$-equivariant isomorphism of $G_{\Q}$-stable subgroups induced by that of Lemma \ref{4-iso} (or $X_i\mapsto X_i^2$ if $C=\e_f$). Then, $\sigma\in\Sha(C^{(-1)}/\Q)$ is visible in $J$ if and only if $\iota_{-1}h$ maps to $0$ in $H^1(\Q,F_f)$ for some lift $h\in H^1(\Q,C^{(-1)}[2^k])$ of $\sigma$---that is, by the standard Kummer exact sequence, if and only if $\iota_{-1}h\in H^1(\Q,F_f[2^k])$ is in the Kummer image of $F_f(\Q)/2^kF_f(\Q)\simeq(\Z/2^k\Z)^r$. Since $h\in\mathrm{Sel}_{2^k}(C^{(-1)}/\Q)$ and the corresponding theorem in the statement shows that $\iota_{-1}$ induces an isomorphism $\mathrm{Sel}_{2^k}(C^{(-1)}/\Q)\simeq\mathrm{Sel}_{2^k}(F_f/\Q)$, we are done by \eqref{selsha}.
\end{proof}

Proposition \ref{visibility}(iii) concerns the visibility of elements of order $4$ in the Tate-Shafarevich group. This is the smallest order for which visibility in a principally polarized abelian surface is not automatic: indeed, Bruin and Fisher \cite{brfi} have shown that it fails in general, due to a local obstruction for K3 surfaces. Therefore, constructing elements of order $4$ that are visible in an abelian surface would be of some interest, especially if one can exhibit an infinite collection of these. For us, this would amount to showing that $\mathrm{rk}(F_f)>\mathrm{rk}(E_f^{(-1)})$ for infinitely many $f$ as in Theorem \ref{4sel}, by Proposition \ref{visibility}(iii). $F_f$ has a nontorsion section, while one easily sees that $E_f^{(-1)}$---seen as a one-parameter elliptic family over $\Q(A)$---has generic rank $0$. Hence one wants to show that, for infinitely many specializations (in a certain subset of $\Z$), its rank does not jump. We will address this specific case in more depth in \S\ref{6.1}. We now give a result about root numbers:
\begin{lem}
     For $A$ as in Theorem \ref{4sel} with $128\mid A+3$, we have:\begin{align}
         w(E_f^{(-1)})=-\mathrm{sgn}(A)(-1)^{\omega_1(A^3+27)}\left(\frac{A}{3}\right),\label{rootcalc}\end{align}
         where $\omega_1(n)$ counts the number of distinct prime factors of $n$ that are $1$ modulo $4$.
         \label{paritylem}
\end{lem}
\begin{proof}
  By the $2$-parity theorem \cite[\S 1]{dok}, $w(E)=1$ exactly when $\dim_{\F_2}\mathrm{Sel}_2(E/\Q)-\dim_{\F_2}E(\Q)[2]$ is even. Our hypotheses make us fall under those of Proposition \ref{2self'}, which implies that the latter quantity is the same for $E=E_f^{(-1)}, \ \e_f^{(-1)}$ (as the $2$-torsions are trivial); so, we find $$w(E^{(-1)}_f)=w(\e^{(-1)}_f).$$ 
  Let $c_6=c_6(\e_f^{(-1)})=-2^63^3$. We use the results of Dokchitser and Cowland-Kellock \cite[\S 2]{dok} to compute the root number $w(\e_f^{(-1)})$: it is the product of local root numbers $w_\ell(\e^{(-1)}_f):=w(\e^{(-1)}_f/\Q_\ell)$ for all places $\ell$ of bad reduction, including $\infty$, where it is always $-1$. They state that $w_2$ is $\left(\frac{-c'_6}{4}\right)$ with $c'_6=c_6/2^{v_2(c_6)}$. In our case, this is $\left(\frac{3^3}{4}\right)=-1$ by Remark \ref{cpcurve}. Therefore, $$w(\e^{(-1)}_f)=\prod_{\ell>2\text{ bad }}w_\ell(\e_f^{(-1)}).$$
  By their Theorem 2.3(iii)-(iv), this is the product of $-\left(\frac{3}{\ell}\right)$ over bad primes $\ell>2$, since they are all multiplicative and the splitting extension is $\Q_\ell(\sqrt{-c_6})=\Q_\ell(\sqrt3)$. By quadratic reciprocity, this is the same as $(-1)^{\frac{\ell+1}{2}}\left(\frac{\ell}{3}\right)$. Since $\ell\mid A^2-3A+9$ implies $(2A-3)^2\equiv-27\pmod\ell$, any such $\ell$ is $1\pmod3$ again by quadratic reciprocity, so \eqref{rootcalc} follows from to the odd-exponents hypothesis of Theorem \ref{4sel}.
\end{proof}
\begin{reml}
    We could have also deduced Lemma \ref{paritylem} from the fact that $c_6(E_f)\in 3\Q_\ell^2$ for $\ell>3$. Thanks to Remark \ref{c4unit}, this gives the more general formula 
    \begin{align}
        w(E_f^{(-1)})=\prod_{\text{odd }\ell\mid B}-\left(\frac{-1}{\ell}\right)\prod_{\text{odd }\ell\mid\Delta}-\left(\frac{3}{\ell}\right),
    \end{align}under the weaker hypotheses that $A$ is odd, coprime with $B$ and $3\nmid AB$.
\end{reml}

\subsection{The subfamily $j=0$}

Over the special fiber $A=0\iff j(Z_f)=0$, the critical families $E_f, E'_f$ blow-up to one-parameter families, as already observed in Proposition \ref{torsj}(ii) for the Prym family. In this case, letting $f_{B}(X)=X^4-4BX$, the aforementioned proposition shows that $E'_B:=(E'_{f_B})^{(-1)}:y^2=x^3+B^2$ is the cubic twist by $B$ of $E'_1:y^2=x^3+1$, while $E_B:=E_{f_B}^{(-1)}:y^2=x^3+27B^4$ by \cite[Proof of Lemma 4.5]{nac}. Here, equality of 2-Selmer groups can still be shown quite easily when $B$ is odd:
\begin{lem}\label{sel}
    For all odd $B\in\Q^\times/(\Q^\times)^3$, we have $\mathrm{Sel}_2(E_B/\Q)\simeq\mathrm{Sel}_2(E'_B/\Q).$
\end{lem}
\begin{proof}
    Identify $H^1(\Q_\ell,E_B[2])\simeq H^1(\Q_\ell,E'_B[2])$ via the (twisted) $4$-isometry of Lemma \ref{4-iso}, which we call $\iota_B$. As before, it suffices to show that the local $2$-Kummer images $E_B(\Q_\ell)/2E_B(\Q_\ell)\subset H^1(\Q_\ell,E_B[2])$ and $E'_B(\Q_\ell)/2E'_B(\Q_\ell)\subset H^1(\Q_\ell,E'_B[2])$ coincide under this identification. We again follow \cite[Proof of Theorem 6.1]{MR}, keeping in mind that---differently from the situation of Proposition \ref{2self'} and Theorem \ref{4sel}---the only bad reduction of our two curves is now additive potentially good (as $j=0$), at $2$, $3$ and primes dividing $B$. So, we can follow verbatim the proof of Mazur and Rubin's \textit{Case 1} and \textit{Case 3}, as well as \textit{Case 2} for $\ell\neq2$, as none of these depends on the missing hypothesis of multiplicative reduction at $p=2$.
    
    For $\ell=2$, we perform $2$-descent for $E_B$ and $E'_B$ over $\Q_2$. As every odd (here we use the hypothesis on the parity of $B$) $2$-adic integer is a cube, over $\Q_2$ we have $E'_B\simeq E_1$ and $E_B\simeq E_1^{(3)},$ so it suffices to perform $2$-descent for \[E_1:y^2=(x+1)(x+\omega)(x+\omega^2),\quad E_1^{(3)}:y^2=(x+3)(x+3\omega)(x+3\omega^2),\]where $\omega$ is a primitive cube root of unity. Let $K=\Q_2(\omega)$ and fix an embedding of $K$ in $\C_2$. Since $\Q_2(E_1[2])=\Q_2(E_1^{(3)}[2])=K$, we have that, for both $C=E_1,E_1^{(3)}$, the Weil pairing induces an injective map \cite[1]{dss}\[w:H^1(\Q_2,C[2])\to H^1(\Q_2,\mu_2(\Q_2)\times\mu_2(K)).\]Composing $w$, to the left with the Kummer map $\kappa:C(\Q_2)/2C(\Q_2)\to H^1(\Q_2,C[2])$ and to the right with the Kummer isomorphism given by Hilbert's Theorem 90, gives a map \begin{align}
        C(\Q_2)/2C(\Q_2)\to\Q_2^\times/(\Q_2^\times)^2\times K^\times/(K^\times)^2\label{x-t}
    \end{align}usually referred to as a $(x-T)$ map \cite[\S 5]{PS}. Indeed, it can be described explicitly \cite[2]{dss}: in our case, it sends the class of $(x,y)\notin C[2]$ to $(x+1,x+\omega)$ for $C=E_1$, and to $(x+3,x+3\omega)$ for $C=E_1^{(3)}$; moreover, it sends the respective $\Q_2$-rational $2$-torsion points $(-1,0)$ and $(-3,0)$ to $\left((-\omega+1)(-\omega^2+1),-1+\omega\right)=(3,\omega-1)$ and $\left((-3\omega+3)(-3\omega^2+3),-3+3\omega\right)=(27,3\omega-3)\sim(3,3\omega-3)$, respectively. Recall that---modulo torsion---we have $C(\Q_2)\simeq\Z_2$ (see for example \cite[VIII]{silv}), so $C(\Q_2)/2C(\Q_2)\simeq(\Z/2)^2$; therefore, it is enough to compute \eqref{x-t} for four classes with distinct images, for each of the two curves.

    We find four such classes for $E_1$ and $E_1^{(3)}$ as follows: we already know that the origin maps to $(1,1)$, as well as where the rational $2$-torsion classes map. Taking $x=4$ and $x=6$ gives points in $E_1(\Q_2)$ whose respective classes map to $(5,\omega+4)$ and $(7,\omega+6)$, by Lemma \ref{ladic}; similarly, taking $x=-7$ and $x=\frac14$ gives two points in $E_1^{(3)}(\Q_2)$ whose respective classes map to $(-4,-7+3\omega)\sim(7,-7+3\omega)$ and $(\frac{13}{4},\frac14+3\omega)\sim(5,\frac14+3\omega)$. The four Kummer representatives obtained for each curve are pairwise distinct, as $1,3,5,7$ are distinct classes in $\Q_2^\times/(\Q_2^\times)^2$.

    It now remains to show that, under the $G_{\Q_2}$-equivariant isomorphism $\gamma:E_1[2]\simeq E^{(3)}_1[2]$ induced by any of the $\iota_B$, these two sets get mapped to one another. Observe that there are just two possible isomorphisms of $E_1[2]$ and $E_1^{(3)}[2]$ as Galois modules, determined by the image of $(-\omega,0)$: $(-3\omega,0)$ or $(-3\omega^2,0)$. Indeed, the two origins and the two rational $2$-torsion points must necessarily map to one another. Since our isomorphism comes from an anti-isometry on the $4$-torsions which is compatible with both (twisted) $2$-congruences by Lemma \ref{compatibility}, we can find in which case we are by checking their actions on Weierstrass points: for $\e_f^{(-1)}[2]\simeq F_f[2]$ we have that $-X_i$ is sent to $-X_i^2$, so in our case $-\omega$ is sent to $-\omega^2$; for $\e_f^{(-1)}[2]\simeq E_f^{(-1)}[2]$, $-X_i$ is sent to $f(X_i)$, so in our case $-\omega$ is sent to $f_1(\omega)=-3\omega$. For the resulting congruence $E'_1[2]\simeq E_1[2]$, $-\omega^2$ is then sent to $-3\omega$. By construction of the $(x-T)$ map, the automorphism of $\Q_2^\times/(\Q_2^\times)^2\times K^\times/(K^\times)^2$ induced by \eqref{x-t} is then just Galois conjugation on the second coordinate. Therefore, it suffices to check that \[\frac{3\omega-3}{-\omega-2},\quad\frac{12\omega+1}{3-\omega},\quad\frac{3\omega-7}{5-\omega}\]are squares in $K^\times$, which is indeed the case.
\end{proof}

\section{Applications}\label{6}
We now give some applications of our results. We start with a few consequences of the local investigation, and move to corollaries of Diophantine and geometric nature later in the section.

\subsection{Equicritical quartics and the Hessian family}\label{6.1}
Looking at the LMFDB pages of various curves $E_f$ for $f$ as in Theorem \ref{companions}, one observes that these curves have not just one rational $3$-isogeny, but \textit{two.} In that case, Theorem \ref{t1} implies that $E_f$ is also a twist of the critical elliptic curve for a quartic $g$ linearly inequivalent to $f$; let us call such $E_f$ an \textit{equicritical} elliptic curve. We have stumbled once again upon the phenomenon of \textit{equicritical polynomials;} inequivalent equicritical quartics have been classified in \cite{nac}, and we indeed see how $f$ as in the statement of Theorem \ref{companions} is the same quartic as $f_t$ in \cite[Theorem 1]{nac}, for $t=-\frac A3$. 

Observe that, because of the transformation law for the discriminant mentioned in the proof of Lemma \ref{crit}, the quartic $g=g_t$ of \cite[Theorem 1]{nac}---which is equicritical to $f=f_t$---has critical elliptic curve \begin{align}
    E_g\simeq E_f^{(-3)}.\label{equitwist}
\end{align}
\begin{es}\label{ex}
    Let $t=-1$ and $f=f_t, \ g=g_t$ with the notation of \cite[Theorem 1]{nac}. Then, $E_f=E_g^{(-3)}$ is the elliptic curve with Cremona label 54a1, while $E'_f$ is 864d1 and $(E'_g)^{(-3)}$ is 1728g1. Notice how the latter ones have different $j$-invariants, $-6^3$ and $-6$ respectively, and are both the only curve in their $\Q$-isogeny class. 
\end{es}

As our Selmer-isomorphism tools only work inside the equicritical family, it is natural to ask if there exists a direct link between equicriticality and Selmer companionship. We have no clear answer to offer, but one potentially relevant observation is that, as remarked in \cite[\S5.1]{nac}, the equicritical elliptic family is the universal family over the modular curve $X_3$ studied by Rubin and Silverberg \cite[\S1.1]{rs}. They show that this universal family is the standard Hesse pencil, which with our parameter choice has the projective model\begin{align}X^3+Y^3+Z^3+AXYZ=0.\label{hesse}\end{align}Since Spencer's $3$-Selmer companion construction \cite{Sp} also involves Hessians, one may hope to tie these instances of preservation of local solubility for covers under various $n$-congruences, to some geometric properties of Hessian pencils in which the curves lie. Still, we remark that the role played by the Hessian in Spencer's family is conceptually orthogonal to that played in ours, since we find a component of the whole family of pairs inside the standard Hesse pencil, while he finds pairs by picking the second curve in the Hesse pencil of the first.

\subsection{Remarks on visibility}\label{subs:visibility}

Notice that, if we could show that $\mathrm{rk}(E_f^{(-1)})=0$ for infinitely many $f$ as in Theorem \ref{4sel}, Proposition \ref{visibility}(iii) would yield infinitely many non-isomorphic elliptic curves with an element of order $4$ in their Tate-Shafarevich group which is visible in a principally polarized abelian surface.
Proving that infinitely many specializations of an elliptic fibration have rank equal to the generic rank is a famously hard problem \cite{zwy2}, even without restricting to integer parameters in a specific congruence class. 

Interestingly, this question has already been considered for the $(-1)$-twist of our family of interest---that is, for the Hesse pencil: Dofs \cite{dofs} identifies the ``probably infinite" set $$\{A\neq1,-5:A^2-3A+9\text{ is prime, }A+3\text{ has no prime factor }\equiv1\text{ mod }3\}$$ as giving rank $0$. In our case, performing $3$-descent and appealing to \cite[Corollary 4.3]{cohenpaz}, one obtains the sufficient condition of $A^3+27$ having no prime factor congruent to $\pm1\pmod{12}$. Even the infinitude of this set appears barely out of reach for current techniques: for instance, by black-boxing a version of \cite[Theorem 1.1]{wuxi} where $p$ is restricted to a congruence class, one can produce infinitely many $A$ such that $A^3+27$ has at most three such prime factors. 

In their work \cite{CM}, Cremona and Mazur examine specific cases of visibility in abelian surfaces, especially in the case where the $n$-congruent elliptic curves are semistable of the same conductor $N$ and the visibility occurs inside of $J_0(N)$. For this reason, taking $n=p$ a prime, they define the (necessary) condition of \textit{modular $p$-congruence}, that is---assuming the curves are \textit{optimal} in their respective isogeny classes---having $E[p]=F[p]$ inside $J_0(N)$. They then remark \cite[Proposition]{CM} that a sufficient condition for this to happen is (together with a mild congruence condition on $v_p((\Delta_E)$ if $p\mid N$, which in our case translates to $v_2(A+3)$ being odd) that the Fourier coefficients $a_m(E), \ a_m(F)$ of the cusp forms associated to $E$ and $F$ agree$\pmod p$ for \textit{all} $m>0$ (as opposed to just those with index coprime to $2N$, which follows just from the $2$-congruence). 

Under the hypotheses of Lemma \ref{semistable}, we are dealing with semistable curves of the same conductor $N=N(A)$. Therefore, it is natural to ask if our $4$-congruence, or at least the $2$-congruence it induces, can be seen inside $J_0(N)$.
Looking at the cusp forms associated to $E_f, E'_f$ for $A=-35, \ 29$ on the LMFDB, we see that the first twenty coefficients---including those divisible by bad primes---agree mod $2$ and even mod $4$, so the aforementioned proposition of Cremona and Mazur hints to a possible affirmative answer. On the other hand, there is the optimality requirement. Notice that, composing the anti-isometry $E'_f[4]\simeq E_f[4]$ with the $3$-isogeny $E_f\to\tilde{E}_f$ to the resolvent curve of Remark \ref{resolvent}, we get an isometry $E'_f[4]\simeq\tilde{E}_f[4]$. There is a third curve in the isogeny class, since $E_f$ has two independent $3$-isogenies. The following lemma will help us in restricting the possibilities for the optimal representatives of the isogeny classes involved in our $4$-congruences:

\begin{lem}Let $f$ be as in Lemma \ref{semistable}.
    \begin{enumerate}
        \item The $\Q$-isogeny class of $E_f$ contains three curves: $E_f,\,\tilde E_f$ and another curve without a $3$-torsion point.
        
        \item $E'_f$ is the only member of its $\Q$-isogeny class.
    \end{enumerate}\label{no-isogenies}
\end{lem}
\begin{proof}
By Lemma \ref{semistable}, every curve isogenous to $E_f$ or $E'_f$ is semistable. Every isogeny is a composite of multiplication maps and isogenies of prime degree $\ell$. Recall that, if $E/\mathbb{Q}$ is semistable and has a rational $\ell$-isogeny, we must have $\ell\mid\#E(\mathbb{F}_p)$ for every prime $p$ of good reduction \cite[proof of Proposition 21]{serre-french}.
    \begin{enumerate}
        \item By Lemma \ref{crit} and the formulas in \cite[Proof of Lemma 4.3]{nac}, $E_f$ has the model $y^2=x^3-2A^2x^2+(A^4+72A)x-8A^3-432$. Modulo $3$ this is the smooth curve $y^2=x^3+x^2+x+A$, with $\#E_f(\mathbb{F}_3)\in\{3,6\}$, so $E_f$ cannot have a rational $\ell$-isogeny with $\ell\geq 5$, and we already know it does not have a $2$-isogeny. Notice that $E_f$ has a rational $3$-torsion point $T$ in virtue of Remark \ref{3pt} and \eqref{equitwist}. By Theorem \ref{t1'}, $\tilde E_f\to E_f$ is the Galois cover $X_f/V_4\to X_f/A_4$, whose Galois group $A_4/V_4$ acts by automorphisms defined over $\mathbb{Q}$; since the deck transformations of an isogeny are the translations by its kernel, the kernel is $\simeq\mathbb{Z}/3$. In virtue of the Weil pairing \cite[Proposition III.8.1]{silv}, the kernel of its dual is then $M\simeq\mu_3$, so $E_f^*=E_f/\langle T\rangle$ is another curve in the isogeny class. 
        
        If $E_f^*$ had a rational $3$-torsion point $S$ (resp$.$ $\tilde E_f$ had a rational $3$-subgroup $N\neq\ker\psi_f$), then $\tilde E_f\to E_f\to E_f^*\to E_f^*/\langle S\rangle$ (resp$.$ $E_f^*\to E_f\to\tilde E_f\to\tilde E_f/N$) would be a cyclic $27$-isogeny, since no step is the dual of the previous one: $\langle T\rangle\neq\ker\hat\psi_f$, $\langle S\rangle\neq \text{im}_{E_f^*}(M)\simeq\mu_3$, and $N\neq\ker\psi=\hat\psi_f(E_f[3])$, forcing the curves to be CM \cite[p. 129]{mazurIsogenies}. On the other hand, in virtue of \eqref{jb=2} $j(E_f)$ cannot be an integer, since $3\nmid A$ and $A^3+27\neq\pm1$; as CM curves have integer $j$-invariants \cite[Theorem II.6.1]{silv-advanced}, we are done.
        
        \item Lemma \ref{prymdata} gives the model $y^2=x^3-Ax^2-4$ for $E'_f$. Since $\Delta(E'_f)=-2^8(A^3+27)$ is prime to $3$ and $\#E'_f(\mathbb{F}_3)\in\{2,3\}$, $E'_f$ has no rational $\ell$-isogeny with $\ell\geq 5$. We also already excluded $2$-isogenies under our hypotheses. The kernel of a rational $3$-isogeny is $\{O,\pm P\}$ with $x(P)=r/s$ a root of the $3$-division polynomial $3x^4-4Ax^3-48x+16A$ \cite[Exercise 3.7]{silv} (and $r,s$ coprime). This gives $4As(r^3-4s^3)=3r(r^3-16s^3)$, forcing $s\mid 3$ and $2\mid r$. Moreover, if an odd prime divides $r^3-4s^3$, it divides $12s^3$ and hence must be $3$, so $r^3-4s^3=\pm 4\cdot 3^b$. Since $3\nmid A$, comparing $3$-adic valuations gives either $s=3$ and $b=0$, or $s=1$ and $b=1+v_3\bigl((r^3-4)-12\bigr)$, which forces $b=2$. Then $r^3=108\pm 4$ or $r^3=4\pm 36$, neither of which is a perfect cube. Hence $E'_f$ has no rational isogeny of prime degree.
    \end{enumerate}
\end{proof}
Thanks to \cite{byeonyhee} and \cite{dummigan}, Lemma \ref{no-isogenies} implies that one of $E_f$ and $\tilde E_f$ is the optimal elliptic curve in their $\Q$-isogeny class. We can now ask:
\begin{q}
    In the situation of Lemma \ref{semistable}, does the $4$-congruence between $E'_f$ and the optimal curve among $E_f$ and $\tilde E_f$ sit inside $J_0(N)$?
\end{q}

\subsection{A variables separated equation}
Equations of the form $$f(X)=g(Y)$$are commonly referred to as \textit{variables separated equations,} and they have a rich literature, see \cite{fried}. In the case $f=g$, letting $d=\deg f$, for $d\ge3$ the equation\begin{align}
    \frac{f(X)-f(Y)}{X-Y}=0\label{varsep}
\end{align}generically describes a smooth, irreducible affine plane curve, whose normalization has genus $\frac{(d-2)(d-3)}{2}$. Equation \eqref{varsep} was studied by Avanzi and Zannier \cite{AZ} in the context of solutions in rational functions. Over number fields, if one is concerned about the infinitude of solutions, the only nontrivial generic case is that of quartics, by Faltings's Theorem. Here we settle it under the assumption that $\sqrt{-1}\in K$:
\begin{thm}Let $f\in K[X]$ be a good quartic and let $h_j(X)$ be any associated precritical sextic, where $j=j(Z_f)$.
    \begin{enumerate}[(i)]
        \item Assume $\sqrt{-1}\in K$. For all but finitely many values of $j\in K$, equation \eqref{varsep} has infinitely many solutions over $K$ for quartics with $j(Z_f)=j$.
        \item If $K=\Q$ and $j(Z_f)\neq0$, equation \eqref{varsep} has infinitely many solutions over $\Q(i)$ unless $f$ is linearly equivalent to $f_{-4}(X)=X^4-8X^2+16X$ or $f_{-8}(X)=X^4-4X^2+4X$, in both of which cases there are two solutions.
        \item There are only finitely many specializations $j\in K$ such that the Pell equation \begin{align}
            R^2-h_jS^2=1
        \end{align}has a nontrivial solution $(R,S)\in \C[X]^2$, but infinitely many such specializations $j\in \C$.
    \end{enumerate}\label{t2}
\end{thm}
\begin{proof}
    We notice preliminarily that a solution to \eqref{varsep} over $L\supset K$ defines a point in $E'_f(L)$, and the converse is true with a finite number of exceptions (coming from the normalization). Therefore, there are infinitely many solutions over $L$ if and only if $\mathrm{rk}_L(E'_f)>0$. 
    \begin{enumerate}[(i)]
        \item We have $\mathrm{rk}_{K}(F_j)>0$ for all but finitely many $j$ by Proposition \ref{torsj}(i) and Silverman's Specialization Theorem \cite[Theorem C.20.3]{silv}, and $F_j\simeq_K E'_j$ since $-1\in (K^\times)^2$, so the claim follows.
        \item If $j\neq0,$ again $\mathrm{rk}_{\Q}(F_j)>0$ unless $u\in\{-8,-4\}$ by Proposition \ref{torsj}(iii), and these values corresponds to quartics linearly equivalent to $f_{-4}$ and $f_{-8}$, given that $u=\frac{A^3}{B^2}$. Since \begin{align*}
            F_{-27648/11}(\Q(i))\simeq\Z/5\text{ and }F_{55296/5}(\Q(i))\simeq\Z/6,
        \end{align*}as can be checked on the respective LMFDB base-change pages (see also Remark \ref{1stcurve}), the finiteness claims follows. As three of the $\Q(i)$-rational points of $E'_j$ come from the normalization at infinity of the model \eqref{varsep}, which is otherwise smooth for good quartics, the solution counts are two and three, respectively. Still, $E'_j$ can have points with $X=Y$ in model \eqref{varsep}, precisely where $f'(X)=0$. As $\frac14f_{-4}'(X)=X^3-4X+4$ is irreducible over $\Q$, while $\frac14f_{-8}'(X)=X^3-2X+1=(X-1)(X^2+X-1)$ has one root $X=1$ in $\Q(i)$, the final counts follow.
        \item It is known that the Pell equation for $h_j$ (both over $\C(j)[X]$ and over $\C[X]$) is solvable if and only if $P_j$ is torsion (as a section or as a point for fixed $j$, respectively), see for instance \cite[\S 1.2]{masserzannier}. Therefore, by Proposition \ref{torsj}(i) and the Silverman specialization Theorem \cite[Theorem C.20.3]{silv}, for each number field $L\supset K$ the Pell equation for $h_j$ is solvable over $\C[X]$ only for finitely many $j\in L$. 
        \vspace{2mm}
    
        \noindent
        On the other hand, applying the result \cite[Theorem 1.1(b)]{masserzannier} of Masser and Zannier with $D=h_j, \ P_D=P_j, \ E_D=F_j, \ n=1$ gives the second part of the statement, provided we show that $F_j$, as an elliptic curve over $K(j),$ is not isogenous to a fixed constant curve $E_0$. If this was the case, it would follow that, with at most finitely many exceptions, all geometric fibers $F_j$ would be mutually isogenous over $\C$---the degree of the isogenies being fixed and equal to that of the isogeny $F_j\sim E_0$. Therefore, their $j$-invariant could assume only finitely many values, contradicting Lemma \ref{prymdata}.
    \end{enumerate}
\end{proof}
\begin{reml}
    For $j=0$ the situation in regards to Theorem \ref{t2}(ii) is mixed: we can get both finitely and infinitely many solutions to \eqref{varsep} over $\Q(i)$, each case for infinitely many quartics. Indeed, the family $E'_{0,B}$ is that of cubic twists of the Mordell curve $E'_{0,1}:y^2=x^3-1$; it is known that infinitely many such twists have $\Q$-rank $1$ \cite[Example 11.3]{krizli}. The same example also gives infinitely many prime cubic twists with $\Q$-rank $0$, while \cite[Corollary 7.7(1)]{cohenpaz} does the same for the $(-1)$-quadratic twist family: infinitude of ranks $0$ over $\Q(i)$ then follows by the Chebotarev density theorem.
\end{reml}
\begin{proof}[Proof of Theorem \ref{varsepbad}]
    The statement for good quartics follows directly from Theorem \ref{t2}(ii) and Remark \ref{badjs}, so we just need to show that \eqref{varsep} has infinitely many solutions over $\Q(i)$ for any bad rational quartic $f$. Remark \ref{bad quartics}(i) tells us what these quartics are those with $j(Z_f)\in\{1728,\infty\}$, so $F_f$ is singular by Lemma \ref{prymdata}. Therefore, the normalization of \eqref{varsep} has genus $0$, and it is enough to show that there is one smooth rational solution. By Remark \ref{bad quartics}(ii), if $j(Z_f)=\infty$ then $f$ is linearly equivalent (over $\Q$) to either $X^4$ or $X^4+X^3$, in which case $(-1,1)$ and $(-1,0)$ are, respectively, smooth solutions. If $j(Z_f)=1728$, then $f$ has the form $X^4+dX^2$ for some some squarefree $d\in\Q^\times$, and again $(-1,1)$ is a smooth solution unless $d=-2$, in which case we can instead take $(-2,2)$.
\end{proof}

\subsection{A model for the Galois closure}

Our cyclic unramified $3$-cover $g_f: X_f\to \c_f$ over $K$ corresponds to the $K$-rational $3$-cyclic subgroup $\ker\widehat\psi_f$ (cf$.$ Proposition \ref{clos}(ii)) embedded in $\jac(\c_f)$ via $\phi_f^*:E_f\hookrightarrow\jac(\c_f)$. Recall from Remark \ref{3pt} that the twist $E_f^{(-3)}$ of $E_f$ by $-3$ has a $3$-torsion point: the base-change $g_f^{(-3)}:X_f\to\c_f^{(-3)}$ of $g_f$ to $K(\sqrt{-3})$ becomes a Kummer cover. 
As Bruin and Flynn remark in \cite[(1)]{BF}, 
curves of genus $g=n-1$ admitting a Kummer $n$-cover (necessarily of genus $g^2$, by the Riemann--Hurwitz formula) over $K$ have a model \begin{align}
    Y^2=G(X)^2+kH(X)^n\label{bruin-flyyn model}
\end{align} with $k\in \mathcal{O}_K^\times, \ G,H\in\mathcal{O}_K^\times[X]$ and $\deg G=n,\ \deg H=2$. 
The cover is then obtained, up to $n$-th twist, by taking an $n$-th root of $Y+G(X)$: in their notation, $$D_{\delta}:\delta U^n=Y+G(X), \quad \delta\in K^\times/(K^\times)^3.$$
Observe that, in our (twisted) case, the curve which Bruin and Flynn denote by $\mathcal{F}$---the quotient of $D_{\delta}$ by any lift of the hyperelliptic involution, when $n-1$ is even---is just $(E'_f)^{(-3)}=F_f^{(3)}$. In \cite[Lemma 2]{BF}, they show that $\jac(D_{\delta})$ splits as $\jac(\c)\times\jac(\mathcal{F})^2$, in accordance with what we already found, and that $D_{\delta}\to\mathcal{F}$ has minimal field of definition $K((-k/\delta^2)^{\frac13})$; since, in our case, $X_f\to (E'_f)^{(-3)}$ is defined over $K(\sqrt{-3}),$ we obtain that the cubic class of $\delta$ is that of $k^{-1}$. 

Given our quartic $f(X)=X^4+2AX^2+4BX$, a direct computation of the model \eqref{bruin-flyyn model} for $\c_f^{(-3)}$ (which follows, for instance, from the formula
$Y^2=\operatorname{Disc}_{X'}\left(\frac{f(X')-f(X)}{X'-X}\right)$ derived in \cite[(14)]{GLNZ} for an affine model of $\c_f$) gives: $$k=-4, \quad G(X)=4(5X^3+9AX+27B), \quad H(X)=-2(X^2+3A).$$
So we find $\delta=2$, and a model for $X_f$ is given by $$U^3G(X)=U^6+H(X)^3,$$with the map $$g_f:X_f\to\c_f, \quad (X,U)\mapsto(X,(2U^3-G(X))/12\sqrt{-3})$$one can now infer from \cite[Lemma 1]{BF}.

\subsection{Jacobians isogenous to the power of an elliptic curve}

The study of decompositions of genus $g$ Jacobians goes back to Serre and Ekedahl \cite{SerEk}, with many open questions still actively researched. Paulhus has studied the problem of determining for which genera $g$ there is a genus-$g$ Jacobian $J=\jac(C)$ with \begin{align}
    J\sim E^g\label{gthpow}
\end{align}for an elliptic curve $E$; in \cite[Theorem 7]{Pau}, she gives one such example for each $g\in\{3,4,5,6\}$, with $C/\Q$ and the isogeny being defined over the field of definition of $\Aut_{\overline\Q}(C)$. Using the results in \cite{swinarski}, one should be able to show that for her specific genus-$4$ example (the curve with automorphism group $(72,40)$), the isogeny is actually defined over $\Q$, and to also recover the isogeny class of $E$. The only instance in the literature of a genus-$4$ curve whose Jacobian splits over $\Q$ as a fourth power of an \textit{explicit} elliptic curve, seems to be \textit{Bring's curve.} Its Jacobian is isogenous to $(50a)^4$, as observed by Serre \cite[\S 8]{serTopGal} \cite{braden-dinsey-hogg}. 
Looking at the splitting \eqref{g4decompose} of the Jacobian of our Galois closure, it would therefore be of some interest to determine for which $j\in\Q$ we have \begin{align}E_j\sim E'_j,\label{factoriso}\end{align} as this would yield other examples of \eqref{gthpow} for $g=4$. Moreover, if the isogeny \eqref{factoriso} is defined over $\Q$ then so is $\jac(X_j)\sim E_j^4$, in virtue of \eqref{g4decompose}.

\begin{reml}
Denoting by $\mathcal A_g$ the moduli space of principally polarized abelian varieties of dimension $g$ and by $\mathcal T_g\subset \mathcal A_g$ the Torelli locus of Jacobians of smooth curves of genus $g$, we have
$$\dim \mathcal A_g=\frac{g(g+1)}2, \qquad \dim \mathcal T_g=3g-3.$$As the representative $E$ for $\overline{\Q}$-isomorphism classes of elliptic curves varies, the abelian varieties $E^g$ form a one-dimensional subvariety of $\mathcal A_g$. Applying any fixed isogeny correspondence to this family gives another one-dimensional subvariety. Any such translate has expected dimension of intersection $1+\dim \mathcal T_g-\dim \mathcal A_g=-\frac{(g-1)(g-4)}2$. So $g=4$ is at the boundary of the unlikely intersection regime, and one expects such a one-dimensional family to meet the Torelli locus in isolated points. Still, there are infinitely many translates, so it is unclear whether we should expect finiteness of such genus-$4$ Jacobians over $\Q$. See \cite{ChenLuZuo} for a discussion of the unlikely intersection case of this problem.
\end{reml}

Observe how a first candidate is the subfamily $f(X)=X^4+4BX$ with $j=0$, for which one has $E_f=E_{0,B}:y^2=x^3-27B^4$ and $E'_f=E_{0,B}':y^2=x^3-B^2,$ respectively by \cite[p.18]{nac} and Proposition \ref{torsj}(ii). A simple computation shows that these two curves are never isogenous over $\Q$, and the field $F$ of smallest degree where
$E_{0,B}\sim E_{0,B}'$ happens for some $B$ 
is either $\Q(i)$ or $\Q(\sqrt3),$ over both of which there is such an isogeny for $B\in(F^\times)^3$. Clearly any such $B$ gives the same pair of isogenous curves; in other words, for $f(X)=X^4+4X$ we have
$\jac(X_f)\sim E_f^4$ over $\Q(i)$ and $\Q(\sqrt3)$, but not over $\Q$. 

The standard strategy for checking whether a genus $2$ correspondence yields isogenous elliptic curves curves over a number field $K$, is to intersect its graph in the affine plane parametrizing the $j$-invariants of the two curves with that of the modular polynomial $\Phi_N$ (for all values of $N$ for which an $N$-isogeny over $K$ can exist), see \cite[Remark 10]{kumar}. For $K=\Q$, the possible values of $N$ are $N\le19, \ N\in\{21,25,27,37,43,67,163\}$. In our case, plugging the $j$-invariants of $E_f$ and $E'_f$ as functions of $j$ (given by \eqref{jcv} and \eqref{jprym}) in $\Phi_N(X,Y)$ for the required values of $N$, yields linear factors for $N=1,3,5,7,9,15,21$, as certified by the code in \cite[Modular isogeny computation]{github}. The corresponding $j$-invariants are listed in the following table (with the exception of $j=0$, where there is always such an isogeny, but as we have seen it is never defined over $\Q$):
\begin{align*}
\begin{array}{c|c|c}
N & j & \bigl(j(E_f),j(E'_f)\bigr) \\
\hline
1 & 2048
  & \left(\dfrac{16384}{5},\dfrac{16384}{5}\right) \\[0.7em]
3 & \dfrac{8000}{7}, \ \dfrac{55296}{5}
  & \left(\dfrac{9938375}{21952},-\dfrac{15625}{28}\right)\text{ and }\left(\dfrac{488095744}{125},\dfrac{16384}{5}\right) \\[0.7em]
5 & -320
  & \left(-\dfrac{121945}{32},-\dfrac{25}{2}\right) \\[0.7em]
7 & \dfrac{3375}{2}
  & \left(-\dfrac{1159088625}{2097152},-\dfrac{140625}{8}\right) \\[0.7em]
9 & -3375
  & \left(-\dfrac{548347731625}{1835008},-\dfrac{15625}{28}\right) \\[0.7em]
15 & 270
  & \left(\dfrac{46969655}{32768},-\dfrac{25}{2}\right) \\[0.7em]
21 & -72000
  & \left(-\dfrac{189613868625}{128},-\dfrac{140625}{8}\right)
\end{array}
\end{align*}
Observe how, even if we are looking at eight pairs of $j$-invariants, that of $E'_j$ only takes four values, being shared by two pairs. This is simply a consequence of the fact that $E_j$ has a $3$-isogeny, which can be composed with that to $E'_j$. The table only tells us for which values of the $j$-invariant $j$ of the critical points we have a $K$-isogeny between $E_f$ and \textit{a twist} of $E'_f$: we have to manually check if the twist is the trivial one. This turns out to happen for three of the pairs:
\begin{prop}
    The only rational values for which there exists an isogeny\begin{align*}
        E_j\sim E'_j.
    \end{align*}over $\Q$ are $j\in\{\frac{8000}{7}, \frac{3375}{2}, 270\}$.
\end{prop}
\begin{proof}
    Computing models for the curves via the implementations in \cite[Critical elliptic curve, Prym elliptic curve]{github}, we find that the three pairs $(E'_j,E_j)$ for $j$ as above are, respectively, those with Cremona labels (126a1, 126a3), (162c2, 162c4) and (50a1, 50a4). The LMFDB confirms that first pair is related by a $3$-isogeny, the second by a $7$-isogeny and the third by a $15$-isogeny.
\end{proof}
Our third pair recovers Bring's curve \cite{braden-dinsey-hogg}, while the other two are, to the best of our knowledge, new. This yields Theorem \ref{fourthsplit}.

\printbibliography
\end{document}